\documentclass[12pt,twoside]{amsart}
\usepackage{amsmath,amsthm,amsmath,amsrefs}

\usepackage{calc}
\usepackage{setspace}
\usepackage{amsbsy,amsmath,amsfonts,amsthm,amssymb,color}
\usepackage[margin=1in]{geometry}
\newtheorem{innercustomtheorem}{Theorem}
\newenvironment{customtheorem}[1]
  {\renewcommand\theinnercustomtheorem{#1}\innercustomtheorem}
  {\endinnercustomtheorem}

\numberwithin{equation}{section}
\theoremstyle{plain}
\newtheorem{theorem}{Theorem}[section]
\newtheorem{lemma}[theorem]{Lemma}
\newtheorem{definition}[theorem]{Definition}
\newtheorem{corollary}[theorem]{Corollary}
\newtheorem{remark}[theorem]{Remark}
\newtheorem{proposition}[theorem]{Proposition}
\newtheorem{example}[theorem]{Example}

\theoremstyle{definition}

\usepackage{mathrsfs}

\newcommand{\ee}{\varepsilon}
\newcommand{\ran}{\text{ran}}
\newcommand{\bN}{\mathbb{N}}

\newcommand{\cN}{\mathcal{N}}

\usepackage{mathtools}

\DeclarePairedDelimiterX\braket[2]{\langle}{\rangle}{#1\,\delimsize\vert\,\mathopen{}#2}

\newcommand{\cB}{\mathcal{B}}

\newcommand{\cK}{\mathcal{K}}

\newcommand{\cO}{\mathcal{O}}
\newcommand{\la}{\langle}
\newcommand{\ra}{\rangle}

\newcommand{\cH}{\mathcal{H}}

\newcommand{\bofh}{\cB(\cH)}

\newcommand{\cM}{\mathcal{M}}
\newcommand{\cU}{\mathcal{U}}
\newcommand{\bC}{\mathbb{C}}
\newcommand{\bR}{\mathbb{R}}

\newcommand{\bZ}{\mathbb{Z}}

\newcommand{\cE}{\mathcal{E}}

\newcommand{\ff}{\mathfrak{f}}

\usepackage{tikz}
\usepackage{tikz-cd}

\title[Self-testing for exact entanglement embezzlement]{Self-testing for exact entanglement embezzlement}

\author{Samuel J. Harris}
\address{Northern Arizona University\\
Department of Mathematics \& Statistics\\
801 S. Osborne Dr., Flagstaff, AZ 86011 USA}
\email{samuel.harris@nau.edu}

\begin{document}
\begin{abstract}
We consider bipartite exact entanglement embezzlement with a catalyst state vector $\psi$ in a Hilbert space $\mathcal{H}$ using unitaries (or more generally, contractions). If $\mathcal{M} \subseteq \mathcal{B}(\mathcal{H})$ is a von Neumann algebra and $U \in M_d \otimes \mathcal{M}$ and $V \in \mathcal{M}' \otimes M_d$ are unitaries (or more generally contractions), then such a protocol is of the form $(U \otimes I_d)(I_d \otimes V)(e_0 \otimes \psi \otimes e_0)=\sum_{i=0}^{d-1} \alpha_i e_i \otimes \psi \otimes e_i$, where each $\alpha_i>0$ and $\sum_{i=0}^{d-1} \alpha_i^2=1$. We show that any such protocol must arise from a unique state on the tensor product $\mathcal{O}_d \otimes \mathcal{O}_d$ of the Cuntz algebra with itself. As a result, we prove that exact entanglement embezzlement is a self-test for a collection of $d$ Cuntz isometries for each party and a unique quasi-free state on the Cuntz algebra $\mathcal{O}_d$ in the sense of \cite{Iz93}. Moreover, we use modular theory to show that the von Neumann algebra generated by the copy of $\mathcal{O}_d$ is the unique separable approximately finite-dimensional Type $\text{III}_{\lambda}$ factor for some $0<\lambda \leq 1$, where $\lambda$ can be determined by an algebraic condition on the Schmidt coefficients of the state $\varphi=\sum_{i=0}^{d-1} \alpha_i e_i \otimes e_i$.
\end{abstract}

\maketitle

\section{Introduction}

Bipartite entanglement embezzlement is a phenomenon in quantum information where two parties (Alice and Bob) use a shared state $\psi$ on a shared resource space and perform quantum operations on their resource space and their own register spaces (without communication) to obtain an entangled state in their shared register space, without disturbing the original shared state. In a sense, the players have ``stolen" entanglement from $\psi$ and now possess it in their own register spaces, without leaving evidence of this act in the state $\psi$. Mathematically, we can describe the process as follows. Suppose the players share a resource Hilbert space $\cH$ that is in the state $\psi$, and that each of their register spaces are dimension $d$. Then embezzling entanglement to obtain an entangled state $\varphi=\sum_{i=0}^{d-1} \alpha_i e_i \otimes e_i$, where $\alpha_0 \geq \cdots \geq \alpha_{d-1} \geq 0$ and $\sum_{i=0}^{d-1} \alpha_i^2=1$, amounts to Alice having a unitary $U \in \cB(\bC^d \otimes \cH)$ and Bob a unitary $V \in \cB(\cH \otimes \bC^d)$ for which $[U \otimes I_d,I_d \otimes V]=0$ and, up to a flip of tensors,
\begin{equation} (U \otimes I_d)(I_d \otimes V)(e_0 \otimes \psi \otimes e_0)=\varphi \otimes \psi. \label{equation: embezzlement}
\end{equation}
(For $\varphi$ to actually be entangled and not separable, we require at least two of the $\alpha_i$'s to be non-zero.) Work of van Dam and Hayden \cite{VH03} shows that one can always do this procedure \textit{approximately} in a finite-dimensional tensor product framework (that is, where $\cH=\cH_A \otimes \cH_B$, $\dim(\cH)<\infty$, and $U$ only acts non-trivially on $\bC^d \otimes \cH_A$ and $V$ on $\cH_B \otimes \bC^d$). It is known that the dimension of the space $\cH$ must tend to infinity as the error in the embezzlement protocol tends to zero \cite{LTW08}. Moreover, even in the case when $\cH=\cH_A \otimes \cH_B$ and $\cH_A$ and $\cH_B$ are infinite-dimensional, achieving (\ref{equation: embezzlement}) exactly is impossible \cite{CLP17}, although it is always possible in a commuting operator framework (see, for example, \cite{CLP17,HP17}).

Recent work has gone into what the observable algebra containing the ``blocks" of $U$ (respectively, blocks of $V$) must look like for both the exact form of (\ref{equation: embezzlement}) and approximate forms of (\ref{equation: embezzlement}). Since $U \in M_d \otimes \bofh$ and $V \in \bofh \otimes M_d$ and $U \otimes I_d$ commutes with $I_d \otimes V$, the matrix blocks of $U$ and $V$ must each $*$-commute with each other; that is, writing $U=\sum_{i,j=0}^{d-1} E_{ij} \otimes U_{ij}$ and $V=\sum_{k,\ell=0}^{d-1} V_{k\ell} \otimes E_{k\ell}$, we must have $[U_{ij},V_{k\ell}]=[U_{ij},V_{k\ell}^*]=0$ \cite{CLP17}. So, it makes sense to consider (\ref{equation: embezzlement}) where $\cM \subseteq \bofh$ is a non-degenerate von Neumann algebra and $U \in M_d \otimes \cM$ and $V \in \cM' \otimes M_d$ are unitaries, and study what properties $\cM$ (and $\cM'$) must have.

Recent work of van Luijk et. al \cite{LSWW24} considers approximate embezzlers. That is, they consider the setting of a von Neumann algebra $\cM \subseteq \bofh$ and a state (unit vector) $\psi \in \cH$ with the property that, for every entangled state $\varphi \in \bC^d \otimes \bC^d$ of full Schmidt rank and for every $\ee>0$, there exist unitaries $U_{\varphi,\ee} \in M_d \otimes \cM$ and $V_{\varphi,\ee} \in \cM' \otimes M_d$ for which (\ref{equation: embezzlement}) holds for $U_{\varphi,\ee}$ and $V_{\varphi,\ee}$ within an error of $\ee$. They proved that in such a setting, $\cM$ is never semifinite as a von Neumann algebra, and when the state induced by $\psi$ is faithful on both $\cM$ and $\cM'$, both algebras are Type $\text{III}$. (In fact, they prove this even when one is able to approximately embezzle a single entangled state in \cite{LSW25}.) Moreover, they proved that a certain function measuring how well embezzlement can be done in a separable approximately finite-dimensional Type $\text{III}$ factor $\cM$ actually classifies the subtype $\lambda \in [0,1]$ of $\cM$, and hence the algebra itself up to isomorphism if $\lambda \neq 0$. They also show that approximate embezzlers exist in all (separable, approximately finite-dimensional, or AFD) Type $\text{III}$ factors with subtype $\lambda$ where $\lambda \neq 0$. (In the case of Type $\text{III}_0$, some factors do exhibit embezzlement, while others perform as badly at embezzlement as semifinite factors \cite{LSWW24}.) One tool used in their study is reducing bipartite (approximate) embezzlement to an appropriate form of monopartite (approximate) embezzlement, which in their case is an approximate unitary equivalence of certain states on $M_d \otimes \cM$ induced by $e_0 \otimes \psi \otimes e_0$ and $\varphi \otimes \psi$. While some of these techniques can be passed to exact embezzlement, others rely crucially on approximation arguments that do not hold in the exact setting.

The purpose of this paper is to provide a link between exact embezzlement of a single state and Cuntz algebras equipped with what are known as quasi-free states. As in the case of approximate embezzlement, the only part of the unitaries $U$ and $V$ that really matter are their first block columns. We prove that achieving (\ref{equation: embezzlement}) for the state $\varphi$ is a self-test for the first column of $U$ and of $V$: up to a compression, they must arise from representations of the Cuntz algebra $\cO_d$ generated by $d$ isometries $V_i$, $i=0,...,d-1$, whose range projections sum to the identity--that is, $V_i^*V_i=\sum_{j=0}^{d-1} V_jV_j^*=I$ for all $i$. The isometries $V_0,...,V_{d-1}$ are the adjoints of the blocks in Alice's first column, after a compression. Although this necessarily forces one to do exact embezzlement with contractions, rather than unitaries, it turns out that one may always arrange to achieve (\ref{equation: embezzlement}) exactly using unitaries in $M_d \otimes \cM$ and $\cM' \otimes M_d$, if one can do it where $U,V$ are contractions (Corollary \ref{corollary: contractions don't generalize monopartite embezzlement}). (This has been known up to dilation of the algebras $\cM$ and $\cM'$ by an argument from \cite{Ha19}, but has not been explicitly carried out using the same algebras $\cM$ and $\cM'$.) Moreover, the state vector $\psi$, the first columns of $U$ and $V$, and the von Neumann algebras they generate, are unique up to a compression and unitary equivalence:

\begin{customtheorem}{\ref{theorem: bipartite yields state on tensor of Cuntz algebra}}
For $a=1,2$, let $(R_a,T_a,\psi_a)$ be a bipartite exact embezzlement protocol for $\varphi$ in $(\cM_a,\cM_a',\cH_a)$. Let $\cN_a$ be the von Neumann subalgebra of $\cM_a$ generated by $\{R_{a,i0}: i=0,...,d-1\}$. Let $P_a$ (respectively, $P_a'$) be the support projection of $\omega_a$ (respectively, $\omega_a'$) on $\cN_a$ (respectively, $\cN_a'$) and let $Q_a=P_aP_a'$. Let $V_{a,i}=Q_aR_{a,i0}^*Q_a$ and $W_{a,j}=Q_aT_{a,j0}^*Q_a$ for $i,j=0,...,d-1$. Then there is a unitary $U:Q_1\cH_1 \to Q_2\cH_2$ satisfying
\begin{enumerate}
\item $U\psi_1=\psi_2$,
\item $UV_{1,j}U^*=V_{2,j}$ and $UW_{1,j}U^*=W_{2,j}$ for all $i,j=0,...,d-1$, and
\item $U(Q_1\cN_1 Q_1) U^*=Q_2\cN_2 Q_2$ and $U(Q_1\cN_1'Q_1)U^*=Q_2 \cN_2' Q_2$.
\end{enumerate}
\end{customtheorem}

Moreover, the state $\psi$ induces a quasi-free state on $\cO_d$ in the sense of Izumi \cite{Iz93}. In fact, if the state $\omega=\la (\cdot)\psi,\psi \ra$ is faithful on $\cM$, then it must satisfy $\omega(V_iXV_j^*)=\delta_{ij}\alpha_i^2 \omega(X)$ for all $i,j=0,...,d-1$ and $X \in \cM$. This is our starting point for defining monopartite exact embezzlement of $\varphi$ (see Definition \ref{definition: monopartite}). We prove that one can go from monopartite exact embezzlement of $\varphi$ in $\cM \subseteq \bofh$ to bipartite exact embezzlement of $\varphi$ (using $\cM,\cM'$) and back (see Theorem \ref{theorem: from monopartite to bipartite}). We also use a bit of modular theory to give an alternate proof that exact embezzlement of a single entangled vector only occurs in Type $\text{III}$ von Neumann algebras when the state $\psi$ is cyclic and separating for $\cM$. Afterward, we use work of Izumi on quasi-free states \cite{Iz93} to determine the von Neumann algebra generated by the Cuntz isometries $V_0,...,V_{d-1}$ for Alice (respectively, Bob) when the state vector $\psi$ is cyclic and separating for the von Neumann algebra generated by $V_0,...,V_{d-1}$. Note that, by Theorem \ref{theorem: from monopartite to bipartite}, we need only consider monopartite exact embezzlement.

\begin{customtheorem}{\ref{theorem: the type}}
Suppose that $(R,\psi)$ exactly embezzles $\varphi$ in $(\cM,\cH)$. Let $\cN$ be the von Neumann algebra generated by $\{R_{i0}: i=0,...,d-1\}$. Let $P$ (respectively, $P'$) be the support projection of the marginal state $\omega$ of $\psi$ in $\cN$ (respectively, of the marginal state $\omega'$ of $\psi$ in $\cN'$), and let $Q=PP'$. Let $G_{\varphi}$ be the closed subgroup of $(\bR^+,\times)$ generated by $\{ \alpha_0^2,...,\alpha_{d-1}^2 \}$.
\begin{enumerate}
\item If $G_{\varphi}$ is countable, then $Q\cN Q$ is isomorphic to the unique (separable) AFD Type $\text{III}_{\lambda}$ factor, where $\lambda=\sup(G_{\varphi} \cap (0,1))=\max(G_{\varphi} \cap (0,1))$. Moreover, $\lambda$ is a root of a polynomial equation of the form $x^{m_0}+\cdots+x^{m_{d-1}}-1=0$ for certain $m_0,...,m_{d-1} \in \bN$.
\item Otherwise, $G_{\varphi}=\bR^+$ and $Q\cN Q$ is isomorphic to the unique (separable) AFD Type $\text{III}_1$ factor.
\end{enumerate}
\end{customtheorem}

We note that, for each $\lambda \in (0,1]$, there is a unique separable AFD Type $\text{III}_{\lambda}$ factor up to isomorphism by deep results of Connes and Haagerup \cite{Co73,Haa87}, and it is given by the Araki-Woods ITPFI (infinite tensor product of finite Type $\text{I}$) factor of Type $\text{III}_{\lambda}$ \cite{AW68}. As a result of Theorem \ref{theorem: the type}, exact embezzlement of $\varphi$ must arise from such an infinite tensor product construction. (In fact, in the case of $d=2$, states on $\cO_2$ yielding Type $\text{III}_{\lambda}$ factors as stated in Theorem \ref{theorem: the type} were constructed in \cite{ALTW91} using infinite tensor products.) Recently, Liu constructed an explicit protocol for simultaneous exact embezzlement of a dense subset of all states in $\bC^{2^n} \otimes \bC^{2^n}$ using an infinite tensor product construction \cite{Li25} (as pointed out in \cite{LSWW24}, the existence of such protocols has been implicitly known since an unpublished manuscript of Haagerup \cite{Haa85} in the language of $\mathbb{Q}$-stable states). Our current work gives strong evidence that this construction is, to some extent, necessary.

As a result of Theorem \ref{theorem: the type}, while approximate embezzlers exist in Type $\text{III}$ factors of all possible subtypes $\lambda \in [0,1]$ \cite{LSWW24}, for exact embezzlement of a single state, the ``smallest" von Neumann algebra possible is Type $\text{III}$ with the subtype $\lambda$ being algebraic (hence belonging to a countable subset of $[0,1]$), and the Type $\text{III}_0$ never appears in this way. Moreover, the types that do appear in this way are determined by roots of polynomials that must have certain properties satisfied ($p(0)=-1$, $p(1)=d-1$, all non-constant coefficients are non-negative, and the exponent list of $p$ is coprime), making it easy to show that infinitely many algebraic $\lambda \in [0,1]$ do \textit{not} appear as ``smallest" observable algebras for embezzlement of some entangled state $\varphi \in \bC^d \otimes \bC^d$ of full Schmidt rank.

It would be interesting to know whether one can exhibit self-testing for other exact embezzlement scenarios--for example, when the same unitaries (or contractions) for Alice and Bob embezzle $e_0 \otimes e_0$ to some entangled state $\varphi$ and some other $e_i \otimes e_j$ to another entangled state $\chi$ that is orthogonal to $\varphi$. To our knowledge, self-testing for multipartite entanglement embezzlement (that is, with three or more parties) of a single entangled state has not yet been explored, although the case of approximate embezzlers in the multipartite setting has been considered by van Luijk, Stottmeister and Wilming \cite{LSW25}. Also, the connection of bipartite exact embezzlement to Cuntz algebras opens up other avenues for research, such as determining what (if any) sort of bipartite entanglement embezzlement phenomenon corresponds to quasi-free states on generalized versions of the Cuntz algebra, such as the Cuntz-Krieger algebra $\cO_A$ given by a $\{0,1\}$-valued matrix. Quasi-free states on Cuntz-Krieger algebras have been studied in work of Okayasu \cite{Ok02} and Kawamura \cite{Ka09}, to name a few. We leave these possibilities for future work.

The remainder of the paper is organized as follows. In section \ref{section: Cuntz}, we provide the links between bipartite exact embezzlement and monopartite exact embezzlement, and show that these always reduce to Cuntz isometries for the two parties with a certain quasi-free condition on the state $\psi$. In section \ref{section: self-test}, we prove the main self-testing results of the paper, namely, that up to a compression, bipartite (respectively, monopartite) exact embezzlement protocols for $\varphi$ must be unitarily equivalent. In section \ref{section: types}, we examine some of the modular theory for exact embezzlement when $\psi$ is cyclic and separating (that is, when $\cM$ is in standard form) and show that $\cM$ is never semifinite, recovering a result from \cite{LSWW24} in the exact case (by alternate means). We also determine the type of the algebra generated by the Cuntz isometries involved.

\section{Entanglement embezzlement and Cuntz algebras}
\label{section: Cuntz}

In this section, we examine exact embezzlement protocols for an entangled state $\varphi \in \bC^d \otimes \bC^d$ of full Schmidt rank, and relate these to Cuntz algebras and quasi-free states in the sense of \cite{Iz93}. We also formulate a monopartite form of exact embezzlement and show that one can pass between the two, without changing the observable algebras involved. 

First, we settle on some notation for this section and for the remainder of the paper. For every Hilbert space $\cH$ discussed below, the inner product $\la \cdot,\cdot \ra$ will be linear in the first variable and conjugate-linear in the second variable. A von Neumann algebra $\cM$ acting on $\cH$ (a unital $*$-closed subalgebra of $\bofh$ that is closed in the weak operator topology) will have commutant denoted by $\cM'=\{T \in \bofh: ST=TS, \, \forall S \in \cM\}$. We also assume some basics in $C^*$-algebra theory; the reader is invited to consult \cite{Da96} for more information.

We will assume throughout that $d \geq 2$ and we will let $\{e_0,...,e_{d-1}\}$ denote the canonical orthonormal basis for $\bC^d$. We fix a state $\varphi=\sum_{i=0}^{d-1} \alpha_i e_i \otimes e_i$ in $\bC^d \otimes \bC^d$, where $\alpha_1 \geq \alpha_2 \geq \cdots \geq \alpha_d>0$ and $\sum_{i=0}^{d-1} \alpha_i^2=1$. Given a Hilbert space $\cH$, we will let $\ff:\bC^d \otimes \cH \otimes \bC^d$ be the flip map of the second and third tensors given by
\[ \ff(e_i \otimes \chi \otimes e_j)=e_i \otimes e_j \otimes \chi, \, \forall i,j=0,...,d-1, \, \chi \in \cH.\]

\begin{definition}
\label{definition: bipartite exact embezzlement}
Let $\cH$ be a Hilbert space and $\cM \subseteq \bofh$ be a non-degenerate von Neumann algebra. A \textbf{(bipartite) exact embezzlement protocol} in $(\cM,\cM',\cH)$ for the state $\varphi$ is a triple $(R,T,\psi)$, where $R \in M_d \otimes \cM$ and $T \in \cM' \otimes M_d$ are contractions and $\psi \in \cH$ is a unit vector satisfying
\begin{equation} (R \otimes I_d)(I_d \otimes T)(e_0 \otimes \psi \otimes e_0)=\ff(\varphi \otimes \psi). \label{equation: bipartite exact embezzlement} \end{equation}
In this case, we say that $(R,T,\psi)$ \textbf{exactly embezzles} $\varphi$ in $(\cM,\cM',\cH)$.
\end{definition}

We note that, if context is clear, then we will drop reference to $(\cM,\cM',\cH)$. Several comments are in order. First, we choose to define bipartite exact embezzlement in terms of contractions, since it is easier to transfer from one exact protocol to another when using contractions. This definition is equivalent to the analogous definition where one only allows $R,T$ to be unitaries, without changing the algebras $\cM$ or $\cM'$ (see Corollary \ref{corollary: contractions don't generalize monopartite embezzlement}). We also note that we are only treating the case when $\varphi$ has full Schmidt rank in $\bC^d \otimes \bC^d$. Indeed, for exact embezzlement protocols for states in $\bC^d \otimes \bC^d$ without full Schmidt rank, the corresponding state $\la (\cdot)\psi,\psi \ra$ on the the block products $R_{ij}T_{k\ell}$ will not be unique, while one can always compress the contractions by using a smaller $d$ that matches with the Schmidt rank of $\varphi$ (see \cite{HP17} for more information). Moreover, we only treat the case where $\varphi$ has Schmidt decomposition with respect to the canonical orthonormal basis for $\bC^d$, since a local unitary transformation of a state in $\bC^d \otimes \bC^d$ will yield a state in the form of $\varphi$. We clarify this last comment in the following proposition.

\begin{proposition}
\label{proposition: basis dependence}
Suppose that a state $\varphi \in \bC^d \otimes \bC^d$ has Schmidt decomposition $\varphi=\displaystyle \sum_{i=0}^{d-1} \alpha_i f_i \otimes g_i$, where $\{f_0,...,f_{d-1}\}$ and $\{g_0,...,g_{d-1}\}$ are two orthonormal bases for $\bC^d$, and where $\alpha_0 \geq \cdots \geq \alpha_{d-1}>0$ are the Schmidt coefficients of $\varphi$. Let $U_f$ (respectively, $U_g$) be the unitary in $M_d$ that sends $f_i$ to $e_i$ for each $i$ (respectively $g_i$ to $e_i$ for each $i$). Let $\cH$ be a Hilbert space and $\cM \subseteq \bofh$ be a von Neumann algebra. Given contractions $R \in M_d \otimes \cM$ and $T \in \cM' \otimes M_d$ and a unit vector $\psi \in \cH$, the following statements are equivalent:
\begin{enumerate}
\item $(R \otimes I_d)(I_d \otimes T)(f_0 \otimes \psi \otimes g_0)=\mathfrak{f}(\varphi \otimes \psi)$.
\item $(\widetilde{R},\widetilde{T},\psi)$ exactly embezzles $\widetilde{\varphi}=\sum_{i=0}^{d-1} \alpha_i e_i \otimes e_i$ in $(\cM,\cM',\cH)$, where $\widetilde{R}=(U_f \otimes I_{\cH})R(U_f^* \otimes I_{\cH})$ and $\widetilde{T}=(I_{\cH} \otimes U_g)T(I_{\cH} \otimes U_g^*)$.
\end{enumerate}
Moreover, if $R_{f,ij}=(f_i^* \otimes \text{id})R(f_j \otimes \text{id})$ and $\widetilde{R}_{e,ij}=(e_i^* \otimes \text{id})\widetilde{R}(e_j \otimes \text{id})$ (respectively, $T_{g,ij}=(g_i^* \otimes \text{id})T(g_j \otimes \text{id})$ and $\widetilde{T}_{e,ij}=(e_i^* \otimes \text{id})\widetilde{T}(e_j \otimes \text{id})$, then $R_{f,ij}=\widetilde{R}_{e,ij}$ and $T_{g,ij}=\widetilde{T}_{e,ij}$. In particular, the von Neumann subalgebra $\cN_f$ of $\cM$ (respectively, $\cN_g$  of $\cM'$) generated by $\{R_{f,i0}:i=0,...,d-1\}$ (respectively, $\{T_{g,i0}:i=0,...,d-1\}$) is unitarily equivalent via $U_f$ (respectively, $U_g$) to the von Neumann subalgebra of $\cM$ (respectively, $\cM'$) generated by $\{\widetilde{R}_{e,i0}: i=0,...,d-1\}$ (respectively, $\{\widetilde{T}_{e,i0}: i=0,...,d-1\}$).
\end{proposition}

\begin{proof}
If (1) holds, then since $U_f^*(e_0)=f_0$ and $U_g^*(e_0)=g_0$,
\begin{align*}
(\widetilde{R} \otimes I_d)(I_d \otimes \widetilde{T})(e_0 \otimes \psi \otimes e_0)&=((U_f \otimes I_{\cH})R \otimes I_d)(I_d \otimes (I_{\cH} \otimes U_g)T)(f_0 \otimes \psi \otimes g_0) \\
&=(U_f \otimes I_{\cH} \otimes U_g)\left(\sum_{i=0}^{d-1} \alpha_i f_i \otimes \psi \otimes g_i \right) \\
&=\sum_{i=0}^{d-1} \alpha_i e_i \otimes \psi \otimes e_i \\
&=\mathfrak{f}(\widetilde{\varphi} \otimes \psi).
\end{align*}
The converse is the same calculation in reverse, since $R=(U_f^* \otimes I_{\cH})\widetilde{R}(U_f \otimes I_{\cH})$ and $T=(I_{\cH} \otimes U_g^*)T(I_{\cH} \otimes U_g)$. 

Lastly, for $i,j=0,...,d-1$ we can write
\begin{align*}
R_{f,ij}&=(f_i^*U_f^* \otimes \text{id})\widetilde{R}(U_f f_j \otimes \text{id}) \\
&=(e_i^* \otimes \text{id})\widetilde{R}(e_j \otimes \text{id})=R_{e,ij}.
\end{align*}
An identical argument holds for $T_{g,ij}$ and $T_{e,ij}$. The claim about subalgebras follows.
\end{proof}

In light of Proposition \ref{proposition: basis dependence}, we will only examine exact entanglement embezzlement in the case when we can write $\varphi=\displaystyle \sum_{j=0}^{d-1} \alpha_j e_j \otimes e_j$, where $\{e_j\}_{j=0}^{d-1}$ is the canonical orthonormal basis for $\bC^d$ and $\alpha_0 \geq \cdots \geq \alpha_{d-1}>0$ satisfy $\displaystyle \sum_{j=0}^{d-1} \alpha_j^2=1$. For simplicity, given an operator $R \in M_d \otimes \cM$ for Alice (respectively, $T \in \cM' \otimes M_d$ for Bob) we will let $R_{ij}=(e_i^* \otimes \text{id})R(e_j \otimes \text{id}) \in \cM$ and $T_{ij}=(\text{id} \otimes e_i^*)T(\text{id} \otimes e_j) \in \cM'$ (that is, dropping the reference to ``$e$"). These operators denote the ``blocks" of the operators $R$ and $T$ with repsect to the canonical basis for $\bC^d$.

Before examining embezzlement protocols further, we will prove two simple propositions that we will use several times throughout.

\begin{proposition}
\label{proposition: equality of CS}
Let $\cH$ be a Hilbert space. Let $X,Y \in \bofh$ be contractions and let $\zeta,\xi \in \cH$ be unit vectors. If $\la X\zeta,Y\xi \ra=1$, then $X\zeta=Y\xi$.
\end{proposition}

\begin{proof}
Using the Cauchy-Schwarz inequality, we see that
\[1=\la X\zeta,Y\xi \ra \leq \|X\zeta\|\|Y\xi\| \leq 1.\]
Thus, the inequalities are equalities. The second inequality being an equality forces $\|X\zeta\|=\|Y\xi\|=1$. The first inequality being an equality implies that there is some $t \in \bC$ such that $t(X\zeta)=Y\xi$. Then $1=\la X\zeta,Y\xi \ra=\la X\zeta,t (X\zeta) \ra=\overline{t} \|X\zeta\|^2=\overline{t}$, so $t=1$. Hence, $X\zeta=Y\xi$. 
\end{proof}

\begin{proposition}
\label{proposition: row contraction on state}
Let $\cH$ be a Hilbert space and $\psi \in \cH$ be a state. Suppose that $A_0,...,A_{d-1} \in \bofh$ are operators such that $\sum_{i=0}^{d-1} A_i^*A_i \leq I$ and $\la A_i^*A_i\psi,\psi \ra=\alpha_i^2$ for each $i$. Then $\sum_{i=0}^{d-1} A_i^*A_i\psi=\psi$.
\end{proposition}

\begin{proof}
The assumption that $\sum_{i=0}^{d-1} A_i^*A_i \leq I$ is equivalent to the assumption that $R=\sum_{i=0}^{d-1} E_{i0} \otimes A_i$ is a contraction in $M_d \otimes \bofh$. Note that $R^*R=E_{00} \otimes \sum_{i=0}^{d-1} A_i^*A_i$, so we obtain
\begin{align*}
1&=\sum_{i=0}^{d-1} \alpha_i^2 \\
&=\sum_{i=0}^{d-1} \la A_i^*A_i\psi,\psi \ra \\
&=\la R^*R(e_0 \otimes \psi),e_0 \otimes \psi \ra.
\end{align*}
Since $R^*R$ is a contraction, an application of Proposition \ref{proposition: equality of CS} shows that $R^*R(e_0 \otimes \psi)=e_0 \otimes \psi$, which forces $\sum_{i=0}^{d-1} A_i^*A_i\psi=\psi$.
\end{proof}

We will first show that bipartite exact embezzlement protocols reduce precisely to the action of the products $R_{i0}T_{j0}$ on the catalyst vector $\psi$. This fact is a slight generalization of those that appear in \cite{CLP17,HP17} (from unitaries to contractions), but we include a proof for the convenience of the reader.

\begin{proposition}
\label{proposition: bipartite and CS}
Suppose that $\mathcal{H}$ is a Hilbert space and that $\cM \subseteq \bofh$ is a von Neumann algebra. Suppose that $R \in M_d \otimes \cM$ and $T \in \cM' \otimes M_d$ are contractions and that $\psi \in \cH$ is a unit vector. The following statements are equivalent:
\begin{enumerate}
\item $(R,T,\psi)$ exactly embezzles $\varphi$.
\item $\la R_{i0}T_{j0}\psi,\psi \ra = \delta_{ij}\alpha_i$ for all $0 \leq i,j \leq d-1$.
\item $R_{i0}T_{j0}\psi=\delta_{ij}\alpha_i\psi$ for all $0 \leq i,j \leq d-1$.
\end{enumerate}
\end{proposition}

\begin{proof}
If (1) holds, then $\displaystyle (R \otimes I_d)(I_d \otimes T)(e_0 \otimes \psi \otimes e_0)=\ff(\varphi \otimes \psi)$. By taking the inner product with $e_i \otimes \psi \otimes e_j$ we have
\[ \delta_{ij}\alpha_i=\la (R \otimes I_d)(I_d \otimes T)(e_0 \otimes \psi \otimes e_0),e_i \otimes\psi \otimes e_j \ra=\la R_{i0}T_{j0}\psi,\psi \ra.\]
Hence, (1) implies (2).

If (2) holds, then $\la R_{i0}T_{j0}\psi,\psi \ra=\delta_{ij}\alpha_i$ for $i,j=0,...,d-1$. Then it follows that
\begin{align*}
1&=\sum_{i=0}^{d-1} \alpha_i^2 \nonumber \\
&=\sum_{i=0}^{d-1} \alpha_i\la R_{i0}T_{i0}\psi,\psi \ra \nonumber\\
&=\sum_{i=0}^{d-1} \alpha_i \la (R \otimes I_d)(I_d \otimes T)(e_0 \otimes \psi \otimes e_0),e_i \otimes \psi \otimes e_i\ra \nonumber \\
&=\sum_{i=0}^{d-1} \la (R \otimes I_d)(I_d \otimes T)(e_0 \otimes \psi \otimes e_0),\alpha_i e_i \otimes \psi \otimes e_i \ra \nonumber \\
&=\la (R \otimes I_d)(I_d \otimes T)(e_0 \otimes \psi \otimes e_0),\ff(\varphi \otimes \psi) \ra.
\end{align*}
By Proposition \ref{proposition: equality of CS}, since $R \otimes I_d$ and $I_d \otimes T$ are contractions, we must have $(R \otimes I_d)(I_d \otimes T)(e_0 \otimes \psi \otimes e_0)=\ff(\varphi \otimes \psi)$. Examining block entries, we see that $\displaystyle \sum_{i,j=0}^{d-1} e_i \otimes R_{i0}T_{j0}\psi \otimes e_j=\ff(\varphi \otimes \psi)$, so $R_{i0}T_{j0}\psi=\delta_{ij}\alpha_i \psi$. Hence, (2) implies (3).

Lastly, if (3) is true, then one readily computes
\[
(R \otimes I_d)(I_d \otimes T)(e_0 \otimes \psi \otimes e_0)=\sum_{i,j=0}^{d-1} e_i \otimes R_{i0}T_{j0}\psi \otimes e_j=\sum_{i=0}^{d-1} \alpha_i e_i \otimes \psi \otimes e_j=\ff(\varphi \otimes \psi),\]
which is the statement of (1).
\end{proof}

We now move to proving the first main result of this section, which is that Bob's operators are uniquely determined on the subspace $\overline{\cM \psi}$ of $\cH$ by Alice's operators. (Similarly, Alice's operators are uniquely determined by Bob's on $\overline{\cM' \psi}$.) Moreover, on these spaces the blocks in the first column of Alice's operator (similarly, Bob's) form a representation of the Cuntz algebra $\cO_d$, which we will see later. Thus, bipartite exact embezzlement protocols for $\varphi$ are partially determined by one player's operators. (The only operators pertinent to bipartite exact embezzlement protocols of a single state are those in the first column for each player, by Proposition \ref{proposition: bipartite and CS}.) 

To accomplish this, we require a few facts about support projections. Given a non-degenerate von Neumann algebra $\cM \subseteq \bofh$ and a state $\psi \in \cH$ with marginal $\omega=\la (\cdot)\psi,\psi \ra$ on $\cM$, the \textbf{support projection of $\omega$} is defined as
\[ s(\omega)=1-\sup \{ X \in \cM: X^2=X=X^*, \, \omega(X)=0\}.\]
For simplicity, the support projection of $\omega$ on $\cM$ will be denoted by $P$ and the support projection of $\omega'$ on $\cM'$ will be denoted by $P'$. We will freely use the fact that $P \in \cM$ and that $\omega$ restricted to the corner algebra $P\cM P$ is faithful; that is, $\omega(X^*X)=0$ for $X \in P\cM P$ implies that $X=0$. As a result, $\psi$ is a separating vector on $P\cM P$: if $X \in P\cM P$ and $X\psi=0$, then $X=0$. Moreover, one can show that $P$ is the orthogonal projection of $\cH$ onto the subspace $\overline{\cM' \psi}$. (Similarly, $P' \in \cM'$ and $\omega'_{|P' \cM' P'}$ is faithful, while $P'$ is the orthogonal projection of $\cH$ onto $\overline{\cM \psi}$.) Lastly, the definition of support projection implies readily that $P\psi=P'\psi=\psi$. We will use these facts freely in the remainder of the paper.

We recall that, given a Hilbert space $\cH$ and an operator $T \in \bofh$, a closed subspace $\cK$ of $\cH$ is called a \textbf{reducing subspace} for $T$ if $T\cK \subseteq \cK$ and $T^* \cK \subseteq \cK$.

\begin{lemma}
\label{lemma: self-testing on subspace}
Suppose that $(R,T,\psi)$ exactly embezzles $\varphi$ in $(\cM,\cM',\cH)$. Let $\omega$ (respectively, $\omega'$) be the marginal induced by $\psi$ on $\cM$ (respectively, $\cM'$), and let $P$ (respectively, $P'$) be the support projection of $\omega$ (respectively, $\omega'$). Set $V_i=PR_{i0}^*P$ and $W_j=P' T_{j0}^*P'$. Then the following hold:
\begin{enumerate}
\item For each $X \in \cM$ and $Y \in \cM'$ and $i,j=0,...,d-1$, 
\begin{equation}
R_{ij}^*(Y\psi)=\delta_{j0}\frac{1}{\alpha_i} YT_{i0}\psi \text{ and } T_{ij}^*(X\psi)=\delta_{j0}\frac{1}{\alpha_i} XR_{i0}\psi. \label{equation: self-testing on subspace}
\end{equation}
In particular, for each $i$, $R_{i0}$ is uniquely determined on $P\cH$ and $T_{i0}$ is uniquely determined on $P'\cH$. Moreover, $P\cH$ is a reducing subspace for $R_{i0}$ and $P'\cH$ is a reducing subspace for $T_{i0}$.
\item For each $i,j=0,...,d-1$, $X \in \cM$ and $Y \in \cM'$,
\begin{equation}
\omega(R_{i0}^*XR_{j0})=\omega(V_iXV_j^*)=\delta_{ij} \alpha_i^2 \omega(X) \label{equation: self-testing quasi-free condition for Alice}
\end{equation}
and
\begin{equation}
\omega'(T_{i0}^*YT_{j0})=\omega'(W_iYW_j^*)=\delta_{ij} \alpha_i^2 \omega'(Y) \label{equation: self-testing quasi-free condition for Bob}
\end{equation}
\item For each $j=0,...,d-1$, we have $V_j^*V_j=\displaystyle \sum_{i=0}^{d-1} V_iV_i^*=P$ and $W_j^*W_j=\displaystyle \sum_{i=0}^{d-1} W_iW_i^*=P'$.
\end{enumerate}
\end{lemma}

\begin{proof}
To show (1), suppose that we know that $R_{ij}^*\psi=\delta_{j0}\frac{1}{\alpha_i} T_{i0}\psi$ and $T_{ij}^*\psi=\delta_{j0}\frac{1}{\alpha_i} R_{i0}\psi$ for all $i,j$. Then for each $X \in \cM$ and $Y \in \cM'$, we will have $R_{ij}^*(Y\psi)=YR_{ij}^*\psi=\delta_{j0}\frac{1}{\alpha_i} YT_{i0}\psi$, and similarly $T_{ij}^*(X\psi)=\delta_{j0}\frac{1}{\alpha_i} XR_{i0}\psi$. Moreover, this shows that $R_{i0}^*(Y \psi) \in \cM'\psi$ for every $Y \in \cM'$, while $R_{i0}(Y\psi)=Y(R_{i0}\psi)=\alpha_i YT_{i0}^*\psi \in \cM' \psi$. Hence, $\overline{\cM'\psi}$ is reducing for $R_{i0}$. (Similarly, $\overline{\cM \psi}$ is reducing for $T_{i0}$.) Thus, we are done the proof of (1) if we can show that $R_{ij}^*\psi=\delta_{j0}\frac{1}{\alpha_i} T_{i0}\psi$ and $T_{ij}^*\psi=\delta_{j0}\frac{1}{\alpha_i} R_{i0}\psi$. 

To this end, since $R \otimes I_d$ and $I_d \otimes T$ commute, we have that

\[1=\la(R \otimes I_d)(I_d \otimes T)(e_0 \otimes \psi \otimes e_0),\ff(\varphi \otimes \psi) \ra=\left\la (R \otimes I_d)(e_0 \otimes \psi \otimes e_0),(I_d \otimes T^*) \left(\ff(\varphi \otimes \psi)\right)\right\ra. \]
Since both of $R \otimes I_d$ and $I_d \otimes T$ are contractions, Proposition \ref{proposition: equality of CS} forces 
\[ (R \otimes I_d)(e_0 \otimes \psi \otimes e_0)=(I_d \otimes T^*)\left(\ff(\varphi \otimes \psi) \right).\]
Since $R$ acts on the first two tensors $\bC^d \otimes \cH$, we can write
\begin{equation}
(R \otimes I_d)(e_0 \otimes \psi \otimes e_0)=\sum_{i,j=0}^{d-1} (E_{ij} \otimes R_{ij} \otimes I_d)(e_0 \otimes \psi \otimes e_0)=\sum_{i=0}^{d-1} e_i \otimes R_{i0}\psi \otimes e_0. \label{equation: R otimes I on start vector}
\end{equation}
On the other hand, since $T^*$ acts on the last two tensors $\cH \otimes \bC^d$,
\begin{equation}
(I_d \otimes T^*)\left( \ff(\varphi \otimes \psi)\right)=\sum_{i,j,k=0}^{d-1} \alpha_i (I_d \otimes T_{kj}^* \otimes E_{jk})(e_i \otimes \psi \otimes e_i)=\sum_{i,j=0}^{d-1} \alpha_i(e_i \otimes T_{ij}^*\psi \otimes e_j). \label{equation: I otimes T* on end vector}
\end{equation}
Comparing (\ref{equation: R otimes I on start vector}) and (\ref{equation: I otimes T* on end vector}), we see that $T_{ij}^*\psi=0$ if $j \neq 0$, while $R_{i0}\psi=\alpha_i T_{i0}^*\psi$ for all $i=0,...,d-1$, so that $T_{ij}^*\psi=\delta_{j0}\frac{1}{\alpha_i} R_{i0}\psi$.

A similar argument involving Proposition \ref{proposition: equality of CS} shows that
\[ (I_d \otimes T)(e_0 \otimes \psi \otimes e_0)=(R^* \otimes I_d)\left(  \ff(\varphi \otimes \psi)\right).\]
We write
\begin{equation}
(I_d \otimes T)(e_0 \otimes \psi \otimes e_0)=\sum_{i,j=0}^{d-1} (I_d \otimes T_{ij} \otimes E_{ij})(e_0 \otimes \psi \otimes e_0)=\sum_{i=0}^{d-1} e_0 \otimes T_{i0}\psi \otimes e_i \label{equation: I otimes T on start vector}
\end{equation}
and
\begin{equation}
(R^* \otimes I_d)\left( \ff(\varphi \otimes \psi)\right)=\sum_{i,j,k=0}^{d-1} \alpha_i(E_{jk} \otimes R_{kj}^* \otimes I_d)(e_i \otimes \psi \otimes e_i)=\sum_{i,j=0}^{d-1} \alpha_i (e_j \otimes R_{ij}^*\psi \otimes e_i). \label{equation: R* otimes I on end vector}
\end{equation}
It follows from (\ref{equation: I otimes T on start vector}) and (\ref{equation: R* otimes I on end vector}) that $T_{i0}\psi=\alpha_i R_{i0}^*\psi$ for $i=0,...,d-1$, while $R_{ij}^*\psi=0$ whenever $j \neq 0$. Hence, $R_{ij}^*\psi=\delta_{ij}\frac{1}{\alpha_i} T_{i0}\psi$, which completes the proof of (1).

To see (2), let $X \in \cM$. Then for $i,j=0,...,d-1$ we have
\begin{align*}
\omega(R_{i0}^*XR_{j0})&=\la XR_{j0}\psi,R_{i0}\psi \ra\\
&=\alpha_j \la T_{j0}^*X\psi,R_{i0}\psi \ra \\
&=\alpha_j \la X\psi,R_{i0}T_{j0}\psi \ra \\
&=\delta_{ij} \alpha_i^2 \la X\psi,\psi \ra \\
&=\delta_{ij} \alpha_i^2 \omega(X),
\end{align*}
where the second line follows by (1) and the fourth line follows by condition (2) of Proposition \ref{proposition: bipartite and CS}. Since $P\cH$ is a reducing subspace for each $R_{i0}$, $P$ commutes with each $R_{i0}$. Using the fact that $P\psi=\psi$, we see that
\[ \omega(V_iXV_j^*)=\la PR_{i0}^*PX(PR_{j0}^*P)^*\psi,\psi \ra=\la R_{i0}^*XR_{j0}\psi,\psi \ra=\delta_{ij} \alpha_i^2 \omega(X).\]
The argument for Bob's operators is similar.

We only show (3) for the Alice's operators $V_j$, as the argument for Bob's operators $W_j$ is similar. Since $P\cH$ is a reducing subspace for $R_{j0}$, the projection $P$ commutes with each $R_{j0}$. Then notice that $V_j^*V_j=(PR_{j0}^*P)^*(PR_{j0}^*P)=R_{j0}R_{j0}^*P$, and similarly $V_iV_i^*=R_{i0}^*R_{i0}P$. So, to determine $V_j^*V_j$ and $\sum_{i=0}^{d-1} V_iV_i^*$, we need only determine $R_{j0}R_{j0}^*P$ and $\sum_{i=0}^{d-1} R_{i0}^*R_{i0}P$, which amounts to computing $R_{j0}R_{j0}^*(Y\psi)$ and $\sum_{i=0}^{d-1} R_{i0}^*R_{i0}(Y\psi)$ for every $Y \in \cM'$, since $\cM'\psi$ is dense in $P\cH$. To this end, given $Y \in \cM$, by (1) we have $R_{j0}^*(Y\psi)=\frac{1}{\alpha_j} YT_{j0}\psi$. Using the fact that $R_{j0}$ commutes with $Y$ and condition (3) of Proposition \ref{proposition: bipartite and CS}, we have
\[ R_{j0}R_{j0}^*(Y\psi)=\frac{1}{\alpha_i} YR_{j0}T_{j0}\psi=Y\psi.\]
By continuity, $R_{j0}R_{j0}^*P=P$, so $V_j^*V_j=P$. Next, using (2), we see that $\sum_{i=0}^{d-1} \la R_{i0}^*R_{i0}\psi,\psi \ra=1$. By Proposition \ref{proposition: row contraction on state}, we obtain $\sum_{i=0}^{d-1} R_{i0}^*R_{i0}\psi=\psi$. It easily follows that $\sum_{i=0}^{d-1} V_iV_i^*(Y\psi)=Y\psi$ for all $Y \in \cM'$. Extending by continuity yields $\sum_{i=0}^{d-1} V_iV_i^*=P$, completing the proof.
\end{proof}

Lemma \ref{lemma: self-testing on subspace} shows that the compression of Alice's operators to the range of $P$ (respectively, the compression of Bob's operators to the range of $P'$) generate a $C^*$-algebra on $P\cH$ (respectively, $P'\cH$) isomorphic to the Cuntz algebra $\cO_d$. For $d \in \bN$ with $d \geq 2$, the \textbf{Cuntz algebra} $\cO_d$ is the universal $C^*$-algebra generated by operators $V_0,...,V_{d-1}$ satisfying $V_i^*V_i=1$ for all $i$ and $\sum_{i=0}^{d-1} V_iV_i^*=1$; that is, each $V_i$ is an isometry, and their range projections sum up to the identity. It is well-known that $\cO_d$ is simple, and hence any collection of operators $V_0,...,V_{d-1} \in \bofh$ satisfying $V_i^*V_i=I$ and $\sum_{i=0}^{d-1} V_iV_i^*=I$ generate a $C^*$-algebra isomorphic to $\cO_d$ \cite{Cu77}. Hence, we will refer to the generators $V_0,...,V_{d-1}$ as a \textbf{collection of $d$ Cuntz isometries}.

Before we turn to the monopartite embezzlement picture, we note that, if Alice and Bob's operators are unitary, then bipartite exact embezzlement yields exact unitary equivalence of certain states involving $e_0$ and $\varphi$, in terms of $\omega$ (respectively, $\omega'$). This recovers unitary equivalence of states from \cite{LSWW24} in the exact case when Alice's operator is unitary. We recall that, if $U$ is a unitary in a von Neumann algebra $\cM$ and if $\chi$ is a continuous linear functional on $\cM$, then $U^*\chi U$ is the state on $\cM$ given by
\[ (U\chi U^*)(X)=\chi(U^*XU).\]

\begin{corollary}
\label{corollary: unitary equivalence of states}

Let $\cM \subseteq \bofh$ be a non-degenerate von Neumann algebra and let $\psi \in \cH$ be a unit vector. Let $(R,T,\psi)$ be a bipartite exact embezzlement protocol for $\varphi$ in $(\cM,\cM',\cH)$. Let $\rho(\varphi)=\displaystyle \sum_{i=0}^{d-1} \alpha_i^2 E_{ii} \in M_d$.
\begin{enumerate}
\item If $R$ is unitary, then 
\begin{equation}
R(\la (\cdot)e_0,e_0\ra \otimes \omega)R^*=\text{Tr}(\rho(\varphi) \cdot) \otimes \omega. \label{equation: Alice unitary equivalence of states}
\end{equation}
\item If $T$ is unitary, then
\begin{equation}
T(\omega' \otimes \la(\cdot)e_0,e_0 \ra)T^*=\omega' \otimes \text{Tr}(\rho(\varphi) \cdot). \label{equation: Bob unitary equivalence of states}
\end{equation}
\end{enumerate}
\end{corollary}

\begin{proof}
If $R$ is unitary and $X=\sum_{i,j=0}^{d-1} E_{ij} \otimes X_{ij} \in M_d \otimes \cM$, then \[R^*XR=\sum_{i,j,k,\ell=0}^{d-1} E_{ij} \otimes R_{ki}^*X_{k\ell}R_{\ell j},\] 
so applying Lemma \ref{lemma: self-testing on subspace},
\[ R(\la (\cdot)e_0,e_0 \ra \otimes \omega)R^*(X)=\la (\cdot)e_0,e_0 \ra \otimes \omega(R^*XR)=\sum_{k,\ell=0}^{d-1} \omega(R_{k0}^*X_{k\ell}R_{\ell 0})=\sum_{k=0}^{d-1} \alpha_k^2 \omega(X_{kk}),\]
and the last quantity is precisely $\text{Tr}(\rho(\varphi) \cdot) \otimes \omega$ applied to $X$. Thus, (\ref{equation: Alice unitary equivalence of states}) holds. The argument for (\ref{equation: Bob unitary equivalence of states}) is similar.
\end{proof}

Similar to work in \cite{LSWW24}, our next goal is to reduce our \textit{bipartite} exact embezzlement protocol to ``monopartite" exact embezzlement protocols on either of the players' observable algebras. While all results that follow could similarly be done for Bob's von Neumann algebra $\cM'$, for simplicity we often only deal with Alice's von Neumann algebra $\cM$. In light of Lemma \ref{lemma: self-testing on subspace}, we make the following definition.

\begin{definition}
\label{definition: monopartite}
Let $\cH$ be a Hilbert space; let $\cM \subseteq \bofh$ be a von Neumann algebra; and let $\psi \in \cH$ be a unit vector with marginal $\omega$ on $\cM$. Let $R \in M_d \otimes \cM$ be a contraction. Then we call $(R,\psi)$ a \textbf{(monopartite) exact embezzlement protocol} in $(\cM,\cH)$ for the state $\varphi$ provided that, for all $X \in \cM$ and $i,j=0,...,d-1$,
\[ \omega(R_{i0}^*XR_{j0})=\delta_{ij} \alpha_i^2 \omega(X).\]
In this case, we say that $(R,\psi)$ \textbf{exactly embezzles} $\varphi$ in $(\cM,\cH)$.
\end{definition}

As with bipartite exact embezzlement, we may drop the reference to $(\cM,\cH)$ if the context is clear. Additionally, the assumption that $R$ is a contraction (and not necessarily a unitary) does not change the definition, as we will see in Corollary \ref{corollary: contractions don't generalize monopartite embezzlement}. Lemma \ref{lemma: self-testing on subspace} shows that, if $(R,T,\psi)$ exactly embezzles $\varphi$ in $(\cM,\cM',\cH)$, then $(R,\psi)$ exactly embezzles $\varphi$ in $(\cM,\cH)$ (similarly, $(T,\psi)$ exactly embezzles $\varphi$ in $(\cM',\cH)$). The goal is to show that monopartite exact embezzlement on one of the algebras $\cM$ or $\cM'$ is enough to recover bipartite embezzlement. This is essentially proven in the following theorem.

\begin{theorem}
\label{theorem: from monopartite to bipartite}
Suppose that $\cM \subseteq \bofh$ is a non-degenerate von Neumann algebra and that $\psi \in \cH$ is a unit vector. Let $R \in M_d \otimes \cM$ be a contraction. The following are equivalent:
\begin{enumerate}
\item $(R,\psi)$ exactly embezzles $\varphi$ in $(\cM,\cH)$.
\item There is a contraction $T \in \cM' \otimes M_d$ for which $(R,T,\psi)$ exactly embezzles $\varphi$ in $(\cM,\cM',\cH)$.
\end{enumerate}
\end{theorem}

\begin{proof}
The proof of $(2) \implies (1)$ is by Lemma \ref{lemma: self-testing on subspace}, so we only prove that $(1) \implies (2)$. Assume that $(R,\psi)$ exactly embezzles $\varphi$ in $(\cM,\cH)$. We let $\omega$ (respectively, $\omega'$) be the marginal state on $\cM$ (respectively, $\cM'$) induced by $\psi$. Let $P$ (respectively, $P'$) be the support projection of $\omega$ in $\cM$ (respectively, $\omega'$ in $\cM'$). For $X \in \cM$ and $i=0,...,d-1$,
\[\left\| \frac{1}{\alpha_i}XR_{i0}\psi \right\|^2=\frac{1}{\alpha_i^2}\la XR_{i0}\psi,XR_{i0}\psi \ra=\frac{1}{\alpha_i^2}\omega(R_{i0}^*X^*XR_{i0})=\omega(X^*X)=\la X\psi,X\psi \ra=\|X\psi\|^2.\]
Thus, for $i=0,...,d-1$, the assignment $W_i:\cM \psi \to \cM \psi$ given by $W_i(X \psi)=\frac{1}{\alpha_i}X R_{i0}\psi$ extends to a well-defined isometry on $P'\cH=\overline{\cM \psi}$, which we also denote by $W_i$.

To determine $W_i^*$, we let $X,Y \in \cM$. Since $(R,\psi)$ exactly embezzles $\varphi$, it follows that
\[ \sum_{i=0}^{d-1} \la R_{i0}^*R_{i0}\psi,\psi \ra=\sum_{i=0}^{d-1} \omega(R_{i0}^*R_{i0})=\sum_{i=0}^{d-1} \alpha_i^2=1.\]
By Proposition \ref{proposition: row contraction on state} we have $\sum_{i=0}^{d-1} R_{i0}^*R_{i0}\psi=\psi$. Then
\begin{align*}
\la W_iX\psi,Y\psi \ra&=\frac{1}{\alpha_i}\la XR_{i0}\psi,Y\psi \ra \\
&=\frac{1}{\alpha_i} \la Y^*XR_{i0}\psi,\psi \ra \\
&=\frac{1}{\alpha_i} \sum_{j=0}^{d-1} \la Y^*XR_{i0}\psi,R_{j0}^*R_{j0}\psi \ra \\
&=\frac{1}{\alpha_i} \sum_{j=0}^{d-1} \omega(R_{j0}^*R_{j0}Y^*XR_{i0}) \\
&=\alpha_i \omega(R_{i0}Y^*X) \\
&=\alpha_i \la X\psi,YR_{i0}^*\psi \ra.
\end{align*}
It follows that $W_i^*(Y\psi)=\alpha_i YR_{i0}^*\psi$ for all $Y \in \cM$. Applying Proposition \ref{proposition: row contraction on state} again, for $X \in \cM$ we compute
\[ \sum_{i=0}^{d-1} W_iW_i^*X\psi= \sum_{i=0}^{d-1} \alpha_i W_iXR_{i0}^*\psi=\sum_{i=0}^{d-1} XR_{i0}^*R_{i0}\psi=X\psi.\]
Extending by continuity, it follows that $\sum_{i=0}^{d-1} W_iW_i^*$ is the identity on $\overline{\cM \psi}=P'\cH$, so that $\sum_{i=0}^{d-1} W_iW_i^*=P'$. (We note that this implies that $W_i^*W_j=0$ whenever $i \neq j$.)  We extend each $W_i$ to all of $\cH$ by defining it to be zero on the subspace $(\cM \psi)^{\perp}=(I-P')\cH$, and we will still have $\sum_{i=0}^{d-1} W_iW_i^*=W_j^*W_j=P'$ for all $j$.

By construction, for every $X,Y \in \cM$ we have $W_i^*(XY\psi)=X(W_i^*Y\psi)=XY(W_i^*\psi)=\alpha_i XYR_{i0}^*\psi$, so extending by continuity shows that $W_i^*X\zeta=XW_i^*\zeta$ whenever $\zeta \in \overline{\cM \psi}=P'\cH$. Since $I-P' \in \cM'$, whenever $\eta \in (I-P')\cH$ and $X \in \cM$ we have $W_i^*X\eta=W_i^*(I-P')X\eta=0$ since $W_i^*(I-P')=0$. Thus, $W_i^*X\eta=XW_i^*\eta=0$ for all $X \in \cM$ and $\eta \in (I-P')\cH$. It follows that each $W_i^*$ (and hence, each $W_i$) belongs to $\cM'$.

Define $T=\sum_{i,j=0}^{d-1} T_{ij} \otimes E_{ij}$ in $\bofh \otimes M_d$ by $T_{ij}=\delta_{j0}W_i^*$ for each $i,j=0,...,d-1$. Each $T_{ij} \in \cM'$, so $T \in \cM' \otimes M_d$. Since $\sum_{i=0}^{d-1} W_iW_i^*=W_j^*W_j=P$ for all $j$, we see that $T=(T_{ij}) \in \cM' \otimes M_d$ is a contraction with $T^*T=P' \otimes E_{00}$ and $TT^*=P' \otimes I_d$. We observe that, since $W_j \in \cM'$ and $T_{j0}=W_j^*$,
\[ \la R_{i0}T_{j0}\psi,\psi \ra=\la R_{i0}\psi,T_{j0}^*\psi \ra=\la R_{i0}\psi,W_j\psi \ra=\frac{1}{\alpha_i} \la R_{i0}\psi,R_{j0}\psi \ra=\frac{1}{\alpha_i} \omega(R_{j0}^*R_{i0})=\delta_{ij}\alpha_i,\]
so $(R,T,\psi)$ exactly embezzles $\varphi$ in $(\cM,\cM',\cH)$ by Proposition \ref{proposition: bipartite and CS}.
\end{proof}

Next, we show that with bipartite and monopartite exact embezzlement protocols, one may always restrict to the case where the state is faithful on each of Alice's and Bob's observable algebras, while being able to arrange for those algebras to be in standard form. These facts will be helpful when considering the modular theory for exact embezzlement protocols in Section \ref{section: types}.

To avoid cumbersome notation, if $R \in M_d \otimes \cM$ and $T \in \cM' \otimes M_d$, and if $P \in \bofh$ is a projection, then we will let $PRP$ denote the operator $(I_d \otimes P)R(I_d \otimes P)$ in $M_d \otimes P\cM P$, and we will let $PTP=(P \otimes I_d)T(P \otimes I_d)$. We also note that, if $\cM \subseteq \bofh$ is a non-degenerate von Neumann algebra $P \in \cM$ is a projection, then $P\cM P$ is a von Neumann algebra on $P\cH$ with commutant $(P\cM P)'=P\cM' P$, and similarly $P'\cM' P'$ is a von Neumann algebra on $P'\cH$ with commutant $(P' \cM' P')'=P'\cM P'$. Moreover, if $Q=PP'$, then $(Q\cM Q)'=Q\cM' Q$. We will use these facts freely below.

\begin{proposition}
\label{proposition: compressed bipartite embezzlement}
Let $\psi \in \cH$ be a unit vector and let $\cM \subseteq \bofh$ be a von Neumann algebra. Suppose that $P \in \cM$ (respectively $P' \in \cM'$) is the support projection of $\omega=\la (\cdot)\psi,\psi \ra$ on $\cM$ (respectively $\omega'=\la (\cdot)\psi,\psi \ra$ on $\cM'$), and let $Q=PP'$. Let $R \in M_d \otimes \cM$ and $T \in \cM' \otimes M_d$ be contractions. The following statements are equivalent:
\begin{enumerate}
\item $(R,T,\psi)$ exactly embezzles $\varphi$ in $(\cM,\cM',\cH)$.
\item $(PRP,PTP,\psi)$ exactly embezzles $\varphi$ in $(PMP,P\cM'P,P\cH)$.
\item $(P'RP',P'TP',\psi)$ exactly embezzles $\varphi$ in $(P'\cM P',P' \cM' P',P'\cH)$.
\item $(QRQ,QTQ,\psi)$ exactly embezzles $\varphi$ in $(Q\cM Q,Q\cM' Q,Q\cH)$.
\end{enumerate}
\end{proposition}

\begin{proof}
We prove that $(1) \implies (2) \implies (4) \implies (1)$; the proof that $(1) \implies (3) \implies (4) \implies (1)$ is almost identical. First, assume $(R,T,\psi)$ exactly embezzles $\varphi$. Recalling that $P\psi=\psi$, since $P \in \cM$ and $T_{j0} \in \cM'$ for each $j$, it follows that for $i,j=0,...,d-1$ we have
\[ \la (PR_{i0}P)(PT_{j0}P)\psi,\psi \ra=\la PR_{i0}T_{j0}\psi,\psi \ra=\la R_{i0}T_{j0}\psi,P\psi\ra=\la R_{i0}T_{j0}\psi,\psi \ra=\delta_{ij}\alpha_i,\]
so $(PRP,PTP,\psi)$ exactly embezzles $\varphi$ by Proposition \ref{proposition: bipartite and CS}. Thus, (1) implies (2). A similar argument with $P'$ shows that (2) implies (4). (Similarly, (1) implies (3) and (3) implies (4).)

Lastly, we show that (4) implies (1). Suppose that $(QRQ,QTQ,\psi)$ exactly embezzles $\varphi$ in $(Q\cM Q,Q\cM' Q,Q\cH)$. Using the facts that $Q=PP'$; $P$ commutes with each $T_{j0}$ and $P'$ commutes with each $R_{i0}$; and $P\psi=P'\psi=\psi$, we obtain
\[ \delta_{ij}\alpha_i=\la (QR_{i0}Q)(QT_{j0}Q)\psi,\psi \ra=\la R_{i0}PP' T_{j0}\psi,\psi \ra=\la P'R_{i0}T_{j0}P\psi,\psi \ra=\la R_{i0}T_{j0}\psi,\psi \ra.\]
By Proposition \ref{proposition: bipartite and CS}, $(R,T,\psi)$ exactly embezzles $\varphi$ in $(\cM,\cM',\cH)$, so (4) implies (1).
\end{proof}

The monopartite version of Proposition \ref{proposition: compressed bipartite embezzlement} is as follows:

\begin{proposition}
\label{proposition: compressed monopartite embezzlement protocol}
Let $\psi \in \cH$ be a unit vector and let $\cM \subseteq \bofh$ be a von Neumann algebra. Suppose that $P \in \cM$ (respectively $P' \in \cM'$) is the support projection of $\omega=\la (\cdot)\psi,\psi \ra$ on $\cM$ (respectively $\omega'=\la (\cdot)\psi,\psi \ra$ on $\cM'$). Let $R \in M_d \otimes \cM$ be a contraction. The following statements are equivalent:
\begin{enumerate}
\item $(R,\psi)$ exactly embezzles $\varphi$ in $(\cM,\cH)$.
\item $(PRP,\psi)$ exactly embezzles $\varphi$ in $(P\cM P,P\cH)$.
\item $(P'RP',\psi)$ exactly embezzles $\varphi$ in $(P'\cM P',P'\cH)$.
\item $(QRQ,\psi)$ exactly embezzles $\varphi$ in $(Q\cM Q,Q\cH)$.
\end{enumerate}
\end{proposition}

\begin{proof}
As in the proof of Proposition \ref{proposition: compressed bipartite embezzlement}, we will show that $(1) \implies (2) \implies (4) \implies (1)$ (the proof that $(1) \implies (3) \implies (4) \implies (1)$ is similar).

Assume (1) holds. If $X \in \cM$, then since $P\cM P \subseteq \cM$, we have
\[ \omega((PR_{i0}P)^*(PXP)(PR_{j0}P)=\la PR_{i0}^*PXPR_{j0}P\psi,\psi \ra=\la R_{i0}^*(PXP)R_{j0}\psi,\psi \ra=\delta_{ij} \alpha_i^2 \omega(PXP).\]
Thus, $(PRP,\psi)$ exactly embezzles $\varphi$ in $(P\cM P,P\cH)$, so (2) holds.

Assuming (2), using the fact that $P'$ commutes with $P$, $R_{i0}^*$ and $R_{j0}$ and that $P\psi=P'\psi=\psi$, for $X \in \cM$ we have that
\begin{align*}
\la (QR_{i0}Q)^*(QXQ)(QR_{j0}Q)\psi,\psi \ra&=\la PR_{i0}^*PP'XPP'R_{j0}P\psi,\psi \ra \\
&=\la PR_{i0}^*PXPR_{j0}P\psi,\psi \ra \\
&=\delta_{ij} \alpha_i^2 \la PXP\psi,\psi \ra \\
&=\delta_{ij}\alpha_i^2 \la QXQ\psi,\psi \ra.
\end{align*}
Thus, $(QRQ,\psi)$ exactly embezzles $\varphi$ in $(Q\cM Q,Q\cH)$, so (2) implies (4).

Lastly, assume (4). Since $P'$ commutes with each block entry of $R$ and $P'\psi=\psi$, it is not hard to see that $(PRP,\psi)$ exactly embezzles $\varphi$ in $(P \cM P,P\cH)$. But then for $X \in \cM$, since $P$ is the support projection of $\omega$ in $\cM$,
\[ \delta_{ij}\alpha_i^2\omega(X)=\delta_{ij}\alpha_i^2\omega(PXP)=\omega((PR_{i0}P)^*(PXP)(PR_{j0}P))=\omega((PR_{i0}P)^*X(PR_{j0}P)).\]
Since $PR_{i0}P \in P\cM P \subseteq \cM$, $(PRP,\psi)$ exactly embezzles $\varphi$ in $(\cM,\cH)$. By Theorem \ref{theorem: from monopartite to bipartite}, there is a contraction $T \in \cM' \otimes M_d$ such that $(PRP,T,\psi)$ exactly embezzles $\varphi$ in $(\cM,\cM',\cH)$. Since $P \in \cM$, it commutes with each $T_{j0}$, so that
\[ \la R_{i0}T_{j0}\psi,\psi \ra=\la PR_{i0}T_{j0}P\psi,\psi \ra=\la (PR_{i0}P)T_{j0}\psi,\psi \ra=\delta_{ij}\alpha_i.\]
Thus, $(R,T,\psi)$ exactly embezzles $\varphi$ in $(\cM,\cM',\cH)$ by Proposition \ref{proposition: bipartite and CS}. By Lemma \ref{lemma: self-testing on subspace}, $(R,\psi)$ exactly embezzles $\varphi$ in $(\cM,\cH)$, which proves (1).
\end{proof}

The utility of Propositions \ref{proposition: compressed bipartite embezzlement} and \ref{proposition: compressed monopartite embezzlement protocol} is that, after compressing the algebras and the space, we may arrange for our exact embezzlement protocol to arise from a unit vector $\psi \in \cH$ whose marginals $\omega$ for $\cM$ and $\omega'$ for $\cM'$ are faithful. Note that having $\omega$ faithful implies that $\psi$ is separating for $\cM$ (equivalently, cyclic for $\cM'$), while $\omega'$ being faithful implies that $\psi$ is separating for $\cM'$ (equivalently, cyclic for $\cM$). Since both $\omega$ and $\omega'$ are faithful, $\psi$ is a cyclic and separating vector for $\cM$ in this case.

Putting results together, we show that every monopartite exact embezzlement protocol amounts to Cuntz isometries with the state $\psi$ playing the role of a ``quasi-free state" (see \cite{Iz93}, for example).

\begin{theorem}
\label{theorem: monopartite must yield Cuntz}
Let $\cM \subseteq \bofh$ be a non-degenerate von Neumann algebra and let $\psi \in \cH$ be a unit vector. Let $\omega$ be the marginal of $\psi$ on $\cM$ with support projection $P$, and let $R \in M_d \otimes \cM$ be a contraction. The following statements are equivalent:
\begin{enumerate}
\item $(R,\psi)$ exactly embezzles $\varphi$ in $(\cM,\cH)$.
\item The operators $V_i=PR_{i0}^*P$, $i=0,...,d-1$, form a collection of $d$ Cuntz isometries in $P\cM P$ satisfying $\omega(V_iXV_j^*)=\delta_{ij}\alpha_i^2\omega(X)$ for all $X \in \cM$.
\end{enumerate}
\end{theorem}

\begin{proof}
By Proposition \ref{proposition: compressed monopartite embezzlement protocol}, $(R,\psi)$ exactly embezzles $\varphi$ in $(\cM,\cH)$ if and only if $(PRP,\psi)$ exactly embezzles $\varphi$ in $(P\cM P,P\cH)$. If $(PRP,\psi)$ exactly embezzles $\varphi$, then by Lemma \ref{lemma: self-testing on subspace} we see that $V_i=PR_{i0}^*P$, $i=0,...,d-1$ define a collection of $d$ Cuntz isometries in $P\cM P$ satisfying $\omega(V_iXV_j^*)=\omega(V_i(PXP)V_j^*)=\delta_{ij}\alpha_i^2 \omega(PXP)=\delta_{ij}\alpha_i^2 \omega(X)$ for all $X \in \cM$. Conversely, if (2) holds, then it readily follows that $(PRP,\psi)$ exactly embezzles $\varphi$, which proves (1) by Proposition \ref{proposition: compressed monopartite embezzlement protocol}.
\end{proof}

\begin{corollary}
\label{corollary: observable algebra is not finite}
If $(R,T,\psi)$ exactly embezzles $\varphi$ in $(\cM,\cM',\cH)$, then neither $\cM$ nor $\cM'$ are finite.
\end{corollary}

\begin{proof}
If $P \in \cM$ is the support projection of the marginal $\omega$ of $\psi$ on $\cM$, then $P\cM P$ contains an isomorphic copy of the Cuntz algebra $\cO_d$, forcing $P$ to be an infinite projection. It follows that $\cM$ is not finite (otherwise all its projections would be finite). The argument for $\cM'$ is similar.
\end{proof}

The same arguments as those made in \cite{LSWW24} show that perfect embezzlement cannot even occur in semifinite factors, using the theory of spectral scales (although they also consider the spectrum of the modular operator in the exact setting). In Section \ref{section: types} we will argue this by way of the modular automorphism group with respect to $\psi$ (when compressing to $P\cM P$ to make $\psi$ faithful). Our argument here that embezzlement cannot occur in finite factors, on the other hand, involves elementary methods and the construction of an embedding of the Cuntz algebra into Alice's (respectively, Bob's) observable algebra.

Before showing that there is no difference between using unitaries and contractions in exact embezzlement protocols, we collect some facts on what the block operators of a monopartite exact embezzlement protocol must look like with respect to the support projection $P$ of $\omega$ in $\cM$.

\begin{proposition}
\label{proposition: P reducing for Ri0s}
Suppose that $(R,\psi)$ exactly embezzles $\varphi$ in $(\cM,\cH)$. Let $P$ be the support projection of the marginal $\omega=\la (\cdot)\psi,\psi \ra$ on $\cM$. Then with respect to the decomposition $\cH=P\cH \oplus (I-P)\cH$, there are operators $B_{ij} \in P\cM(I-P)$ and $C_{ij} \in (I-P)\cM(I-P)$ such that
\[ R_{ij}=\begin{pmatrix} \delta_{j0} V_i^* & 0 \\  (1-\delta_{j0})B_{ij} & C_{ij} \end{pmatrix}, \, i,j=0,...,d-1.\]
In particular, $P$ is a central projection in $W^*(\{R_{i0}\}_{i=0}^{d-1})$ and we have $PW^*(\{R_{i0}\}_{i=0}^{d-1})P=W^*(\{PR_{i0}P\}_{i=0}^{d-1})$.
\end{proposition}

\begin{proof}
We consider the decomposition $\cH=\text{ran}(P) \oplus \text{ran}(P)^{\perp}$. For each $i,j$ we write $R_{ij}=\begin{pmatrix} Z_{ij} & A_{ij} \\ B_{ij} & C_{ij} \end{pmatrix}$. First, note that $V_i=Z_{i0}^*$, $i=0,...,d-1$, is a collection of $d$ Cuntz isometries in $P\cM P$ by Theorem \ref{theorem: monopartite must yield Cuntz}. Since $R$ is a contraction, $RR^* \leq I_d$ and $R^*R \leq I_d$. Considering the inequality $RR^* \leq I_d$, the $(i,i)$ block entry must satisfy $\sum_{j=0}^{d-1} R_{ij}R_{ij}^* \leq I$. With respect to $\cH=\ran(P)\oplus\ran(P)^{\perp}$ the upper-left corner of $\sum_{j=0}^{d-1} R_{ij}R_{ij}^*$ is 
\[ P(\sum_{j=0}^{d-1} R_{ij}R_{ij}^*)P=\sum_{j=0}^{d-1} (Z_{ij}Z_{ij}^*+A_{ij}A_{ij}^*) \leq P.\]
But $Z_{i0}Z_{i0}^*=V_i^*V_i=P$, so we must have
\[ \sum_{j=1}^{d-1} Z_{ij}Z_{ij}^*+\sum_{j=0}^{d-1} A_{ij}A_{ij}^*=0.\]
It follows that $A_{ij}=0$ for all $i,j$ and that $Z_{ij}=\delta_{j0}V_i^*$. Next, since $R^*R \leq I_d$, compressing to the $(0,0)$ block entry yields $\sum_{i=0}^{d-1} R_{i0}^*R_{i0} \leq I$. With respect to $\cH=\text{ran}(P) \oplus \text{ran}(P)^{\perp}$, the upper-left corner of $\sum_{i=0}^{d-1} R_{i0}^*R_{i0}$ is 
\[P(\sum_{i=0}^{d-1} R_{i0}^*R_{i0})P=\sum_{i=0}^{d-1} (V_iV_i^*+B_{i0}^*B_{i0})=P+\sum_{i=0}^{d-1} B_{i0}^*B_{i0}.\] 
This forces $B_{i0}=0$ for all $i$. Hence, with respect to the decomposition $\ran(P) \oplus \ran(P)^{\perp}$, each $R_{i0}$ is diagonal, while $R_{ij}=\begin{pmatrix} 0 & 0 \\ B_{ij} & C_{ij} \end{pmatrix}$ if $j \neq 0$. As $\text{ran}(P)$ is a reducing subspace for each $R_{i0}$ by Lemma \ref{lemma: self-testing on subspace} and since $P \in \cM$, we see that $P$ is central in $\cM_0=W^*(\{R_{i0}\}_{i=0}^{d-1})$. Then $P\cM_0 P$ is the von Neumann subalgebra of $P\cM P$ generated by $\{V_i\}_{i=0}^{d-1}$, and hence $P\cM_0 P$ is the von Neumann algebra generated by $\{PR_{i0}P\}_{i=0}^{d-1}$.
\end{proof}

Throughout this section, we have allowed the flexibility of Alice and Bob using contractions with block entries in their observable algebras. Now, we can show that this did not add any generality. However, in passing to unitaries one necessarily loses faithfulness of the state involved. The key tools in the proof are the Halmos dilation of a contraction in $\cM$ to a unitary in $M_2(\cM)$, and the fact that $M_2(\cM) \simeq \cM$ whenever $\cM$ is a von Neumann algebra that is not finite.

\begin{corollary}
\label{corollary: contractions don't generalize monopartite embezzlement}
Let $\cM \subseteq \bofh$ be a von Neumann algebra and let $\psi \in \cH$ be a unit vector.
\begin{enumerate}
\item If $(R,\psi)$ exactly embezzles $\varphi$ in $(\cM,\cH)$, then there is a unitary $U \in M_d \otimes \cM$ and a unit vector $\zeta \in \cH$ such that $(U,\zeta)$ exactly embezzles $\varphi$.
\item If $(R,T,\psi)$ exactly embezzles $\varphi$ in $(\cM,\cM',\cH)$, then there are unitaries $U \in M_d \otimes \cM$ and $V \in \cM' \otimes M_d$, along with a unit vector $\eta \in \cH$ so that $(U,V,\eta)$ exactly embezzles $\varphi$ in $(\cM,\cM',\cH)$.
\end{enumerate}
\end{corollary}

\begin{proof}
Let $P \in \cM$ be the support projection of $\omega=\la (\cdot)\psi,\psi \ra$ on $\cM$. Then $V_i=PR_{i0}^*P$ defines a collection of $d$ Cuntz isometries in $P\cM P$ by Theorem \ref{theorem: monopartite must yield Cuntz}, such that $\omega(V_iXV_j^*)=\delta_{ij}\alpha_i^2 \omega(X)$ for all $i,j=0,...,d-1$ and $X \in \cM$. Note that $P\cM P$ is not finite by Corollary \ref{corollary: observable algebra is not finite}. By the projection halving lemma (see \cite[Lemma~6.3.3]{KR86}), we may find operators $G_0,G_1 \in P\cM P$ such that $G_i^*G_i=P$ for $i=0,1$, while $G_0G_0^*$ and $G_1G_1^*$ are non-zero orthogonal projections summing to $P$, forcing $G_0^*G_1=G_1^*G_0=0$. (In other words, $\{G_0,G_1\}$ is a collection of $2$ Cuntz isometries in $P\cM P$.) Let $\widetilde{R}=(I_d \otimes P)R(I_d \otimes P)=(PR_{ij}P)_{i,j=0}^{d-1}$ and write $\widetilde{R}_{ij}=PR_{ij}P$. Note that $\widetilde{R}_{ij}=\delta_{j0}V_i^*$ by Proposition \ref{proposition: P reducing for Ri0s}, so that $\widetilde{R}\widetilde{R}^*=I_d \otimes P$ and $\widetilde{R}^*\widetilde{R}=E_{00} \otimes P$. So $\widetilde{R}$ has Halmos dilation 
\[\begin{pmatrix} \widetilde{R} & \sqrt{I_d \otimes P-\widetilde{R}\widetilde{R}^*} \\ \sqrt{I_d \otimes P-\widetilde{R}^*\widetilde{R}} & -\widetilde{R}^* \end{pmatrix}=\begin{pmatrix} \widetilde{R} & 0 \\ \sum_{j=1}^{d-1} E_{jj} \otimes P & -\widetilde{R}^*\end{pmatrix},\] which is unitary in $M_2(M_d(P\cM P))$. A canonical shuffle $M_2(M_d(P\cM P)) \to M_d(M_2(P\cM P))$ (see \cite{Pa02}) yields a unitary $\widetilde{U}=\begin{pmatrix}\begin{pmatrix} \widetilde{R}_{0,0,i,j} & \widetilde{R}_{0,1,i,j} \\ \widetilde{R}_{1,0,i,j} & \widetilde{R}_{1,1,i,j} \end{pmatrix}\end{pmatrix}_{i,j=0}^{d-1}$, where $\widetilde{R}_{0,0,i,j}=\delta_{j0}V_i^*$, $\widetilde{R}_{0,1,i,j}=0$, $\widetilde{R}_{1,0,i,j}=(1-\delta_{i0})\delta_{ij} P$, and $\widetilde{R}_{1,1,i,j}=-\widetilde{R}_{ji}^*=-\delta_{i0}V_j$. Notice that since $\widetilde{U}$ (and the Halmos dilation of $\widetilde{R}$) are unitary, for all $i,j=0,...,d-1$ and $a,b=0,1$ we have \[\sum_{c=0}^1\sum_{k=0}^{d-1} \widetilde{R}_{a,c,i,k}\widetilde{R}_{b,c,j,k}^*=\sum_{c=0}^1\sum_{k=0}^{d-1} \widetilde{R}_{c,a,k,i}^*\widetilde{R}_{c,b,k,j}=\delta_{ab}\delta_{ij}P.\] 
Then define, for $i,j=0,...,d-1$,
\[ U_{ij}=\sum_{a,b=0}^1 G_a \widetilde{R}_{a,b,i,j}G_b^*+\delta_{ij}(I-P) \in \cM.\]
The extra term $\delta_{ij}(I-P)$ just extends each $U_{ij}$ to be defined on all of $\cH$, in such a way that $U_{ij} \in \cM$ for all $i,j$. Note that each $\widetilde{R}_{a,b,i,j} \in P\cM P$, so that $P\widetilde{R}_{a,b,i,j}=\widetilde{R}_{a,b,i,j}P=\widetilde{R}_{a,b,i,j}$ for all $a,b=0,1$ and $i,j=0,...,d-1$, while $G_a^*(I-P)=(I-P)G_a=0$ for $a=0,1$ since $G_a \in P\cM P$. Thus, one can compute, for each $i,j=0,...,d-1$,
\begin{align*}
\sum_{k=0}^{d-1} U_{ik}U_{jk}^*&=\sum_{a,b,c,d=0}^1 \sum_{k=0}^{d-1} (G_a \widetilde{R}_{a,b,i,k}G_b^* G_d \widetilde{R}_{c,d,j,k}^* G_c^*+\delta_{ik}\delta_{jk} (I-P)) \\
&=\sum_{a,b,c=0}^1 \sum_{k=0}^{d-1} G_a \widetilde{R}_{a,b,i,k}P \widetilde{R}_{c,b,j,k}^* G_c^* + \delta_{ij} (I-P) \\
&=\sum_{a,b,c=0}^1 \sum_{k=0}^{d-1} G_a \widetilde{R}_{a,b,i,k}\widetilde{R}_{c,b,j,k}^* G_c^* + \delta_{ij} (I-P) \\
&=\sum_{a=0}^1 \delta_{ij}G_a G_a^* + \delta_{ij}(I-P) \\
&=\delta_{ij}I.
\end{align*}

A similar calculation shows that $\sum_{k=0}^{d-1} U_{ki}^*U_{kj}=\delta_{ij}I$ for all $i,j=0,...,d-1$. Thus, $U$ is unitary in $M_d \otimes \cM$. Define $\zeta=G_0\psi$. Note that $G_0^*G_0=P$ and $P\psi=\psi$, we have $\| G_0\zeta\|^2=\la G_0^*G_0\psi,\psi \ra=\la P\psi,\psi \ra=1$, so $\zeta$ is a unit vector in $\cH$. As $R_{1,0,i,0}=0$ for all $i$, it follows that
\[U_{i0}G_0=\sum_{a,b=0}^1 G_a \widetilde{R}_{a,b,i,0}G_b^*G_0+\delta_{i0}(I-P)G_0=\sum_{a=0}^1 G_a R_{a,0,i,0}=G_0V_i^*.\]
Taking adjoints, $G_0^*U_{i0}^*=V_iG_0^*$. Then for $i,j=0,...,d-1$ and $X \in \cM$, we compute
\begin{align*}
\la U_{i0}^*XU_{j0} \zeta,\zeta \ra&=\la G_0^*U_{i0}^*XU_{j0}G_0\psi,\psi \ra \\
&=\la V_i(G_0^*XG_0)V_j^*\psi,\psi \ra \\
&=\delta_{ij} \alpha_i^2 \la G_0^*XG_0\psi,\psi \ra \\
&=\delta_{ij} \alpha_i^2 \la X\zeta,\zeta \ra.
\end{align*}
Thus, $(U,\zeta)$ is a monpartite exact embezzlement protocol for $\varphi$ in $(\cM,\psi)$, establishing (1). Moreover, if $T \in \cM' \otimes M_d$ is a contraction such that $(R,T,\psi)$ exactly embezzles $\varphi$ in $(\cM,\cM',\cH)$, then defining $U$ in this way and noting that $[T_{ij},G_a]=0$ for all $i,j=0,...,d-1$ and $a=0,1$, we obtain
\begin{align*}
\la U_{i0}T_{j0}\zeta, \zeta \ra&=\la G_0^*U_{i0}T_{j0}G_0\psi,\psi \ra \\
&=\la G_0^*U_{i0}G_0T_{j0}\psi,\psi \ra \\
&=\la G_0^*G_0V_i^*T_{j0}\psi,\psi \ra \\
&=\la R_{i0}T_{j0}\psi,\psi \ra \\
&=\delta_{ij} \alpha_i.
\end{align*}
By Proposition \ref{proposition: bipartite and CS}, $(U,T,\zeta)$ exactly embezzles $\varphi$. By performing the same proof but for $\cM'$ instead of $\cM$, we can find a unitary $V \in \cM' \otimes M_d$ and a unit vector $\eta \in \cH$ such that $(U,V,\eta)$ is a bipartite exact embezzlement protocol in $(\cM,\cM',\cH)$.
\end{proof}

\section{Self-testing}
\label{section: self-test}

In this section, we show that, after compressing by support projections and considering the von Neumann algebras generated by the operators actually involved, any two bipartite exact embezzlement protocols for a state $\varphi \in \bC^d \otimes \bC^d$ of full Schmidt rank must be unitarily equivalent. This unitary equivalence sends the catalyst vector from one protocol to the other, and the equivalence passes to the von Neumann algebra generated by Alice's collection of Cuntz isometries (respectively, Bob's isometries). Our main tool is the uniqueness, up to unitary equivalence, of the GNS representation of a state on a unital $C^*$-algebra. In the case of embezzling $\varphi$, the catalyst vector $\psi$ induces a state on the tensor product $\cO_d \otimes_{\max} \cO_d$ of the Cuntz algebra $\cO_d$ with itself. Since $\cO_d$ is a nuclear $C^*$-algebra, all $C^*$-norms on the algebraic tensor product $\cO_d \otimes \cO_d$ agree; that is, $\cO_d \otimes_{\min} \cO_d=\cO_d \otimes_{\max} \cO_d$.

We settle on some notation for this section. When dealing abstractly with $\cO_d$ (that is, without a representation in mind), we will let $v_0,...,v_{d-1}$ denote a (fixed) collection of $d$ Cuntz isometries that generate $\cO_d$ as a $C^*$-algebra. When considering $\cO_d$ concretely as a subalgebra of $\bofh$, we will use $V_0,...,V_{d-1}$ to denote a fixed collection of $d$ Cuntz isometries generating (the isomorphic copy of) $\cO_d$ in $\bofh$. For $m \in \bN$ and $\mu \in \{0,...,d-1\}^m$, we define $V_{\mu}=V_{\mu_1}V_{\mu_2} \cdots V_{\mu_m}$. We will often refer to the string $\mu=\mu_1 \cdots \mu_m$ as a finite (non-empty) word in $\{0,...,d-1\}$. We use the convention that $V_{\emptyset}=I$. Given two finite words $\mu,\nu$ in $\{0,...,d-1\}$, we define $\mu\nu$ to be the concatenation of $\mu$ and $\nu$; if $\mu=\mu_1 \cdots \mu_m$ and $\nu=\nu_1 \cdots \nu_n$, then $\mu \nu=\mu_1 \cdots \mu_m \nu_1 \cdots \nu_n$. (If $\mu=\emptyset$, then $\mu\nu=\nu$, and similarly if $\nu=\emptyset$, then $\mu\nu=\mu$.) We then note that for any two finite words $\mu,\nu$ in $\{0,...,d-1\}$, we have $V_{\mu}V_{\nu}=V_{\mu \nu}$ and $V_{\mu}^*V_{\nu}^*=(V_{\nu}V_{\mu})^*=V_{\nu \mu}^*$. We note that all of this notation can be used when thinking about the Cuntz algebra $\cO_d$ abstractly with the abstract generators $v_0,...,v_{d-1}$. 

We will use similar notation for products of Schmidt coefficients of our state $\varphi \in \bC^d \otimes \bC^d$ of full Schmidt rank, and write $\alpha_{\mu}=\alpha_{\mu_1} \cdots \alpha_{\mu_m}$ for any finite word $\mu=\mu_1 \cdots \mu_m$ in $\{0,...,d-1\}$. (While we will not need it, we use the convention that $\alpha_{\emptyset}=1$.) A well-known fact (see \cite{Cu77}) is that one can write
\[ \cO_d=\overline{\text{span}}\{V_{\mu}V_{\nu}^*: \mu \in \{0,...,d-1\}^m, \, \nu \in \{0,...,d-1\}^n, \, m,n \in \bN\}.\]
Hence, behavior of states on $\cO_d$ (and later, on $\cO_d \otimes_{\max} \cO_d$) can be reduced to their behavior on such words.

We also will have occasion to consider similar words in the operators $R_{i0}$ of a monopartite exact embezzlement protocol for $\varphi$; however, these words will be flipped from the words involving Cuntz isometries. To simplify notation, given a tuple $\mu \in \{0,...,d-1\}^m$, we will define $R_{\mu}=R_{\mu_1,0} \cdots R_{\mu_m,0}$, with the same product and concatenation rule as above for Cuntz isometries.

\begin{lemma}
\label{lemma: state on Alice's Od with no stars}
Let $(R,\psi)$ be a monopartite exact embezzlement protocol for $\varphi$ in $(\cM,\cH)$, Let $\omega$ be the marginal of $\psi$ on $\cM$. If $\mu$ is a finite non-empty word in $\{0,...,d-1\}$, then $\omega(R_{\mu})=0$.
\end{lemma}

\begin{proof}
Write $\mu=\mu_1 \cdots \mu_m$ for some $m \in \bN$. Let $P$ be the support projection of $\omega$ on $\cM$. By Proposition \ref{proposition: P reducing for Ri0s}, $P$ commutes with $R_{i0}$ for $i=0,...,d-1$. Letting $V_i=PR_{i0}^*P$, and noting that $\omega(PXP)=\omega(X)$ for all $X \in \cM$, we see that
\[ \omega(R_{\mu})=\omega((PR_{\mu_1,0}P) \cdots (PR_{\mu_m,0}P))=\omega(V_{\mu_1}^* \cdots V_{\mu_m}^*)=\omega(V_{\overline{\mu}}^*)=\overline{\omega(V_{\overline{\mu}})},\]
where $\overline{\mu}$ denotes the word $\mu$ backwards (that is, $\overline{\mu}=\mu_m \cdots \mu_1$). Hence, to show the result, it suffices to prove it in the case when $\omega$ is faithful and $V_i=R_{i0}^*$ for all $i$, by Proposition \ref{proposition: compressed monopartite embezzlement protocol}.

Note that, by Theorem \ref{theorem: monopartite must yield Cuntz}, for each $i$ we have $\omega(V_i)=\sum_{j=0}^{d-1} \omega(V_iV_jV_j^*)=\alpha_i^2 \omega(V_i)$. As $\alpha_i>0$, we see that $\omega(V_i)=0$.
Now, if $\nu \in \{0,...,d-1\}^n$ and $n>1$, write $\nu=\nu_1 \cdots \nu_n$. We note that
\[ \omega(V_{\nu})=\sum_{j=0}^{d-1} \omega(V_{\nu_1} \cdots V_{\nu_n}V_jV_j^*)=\alpha_{\nu_1}^2 \omega(V_{\nu_2} \cdots V_{\nu_n}V_{\nu_1}),\]
by Theorem \ref{theorem: monopartite must yield Cuntz}. Iterating this process $n$ times, we see that 
\[ \omega(V_{\nu})=\alpha_{\nu_1}^2 \cdots \alpha_{\nu_n}^2 \omega(V_{\nu_1} \cdots V_{\nu_n})=\alpha_{\nu}^2 \omega(V_{\nu}).\]
Since $\alpha_{\nu}^2>0$, we must have $\omega(V_{\nu})=0$. Since this holds for all non-empty words in $\{0,...,d-1\}$, we see that $\omega(R_{\mu})=0$ for all non-empty words $\mu$ in $\{0,...,d-1\}$.
\end{proof}

Next, we prove that the state on Alice's copy of $\cO_d$ induced by the catalyst vector $\psi$ in exact embezzlement of $\varphi$ is unique, which may be of independent interest.

\begin{theorem}
\label{theorem: unique state on Alice's Od}
There is a unique state $s$ on $\cO_d$ such that $s(v_ixv_j^*)=\delta_{ij}\alpha_i^2 s(x)$ for all $x \in \cO_d$. Moreover, if $(R,\psi)$ is a monopartite exact embezzlement protocol for $\varphi$ in $(\cM,\cH)$ and if the marginal $\omega$ of $\psi$ is faithful on $\cM$, then for the unital $*$-homomorphism $\pi:\cO_d \to \cM$ given by $\pi(v_j)=V_j$, we have
\[ s(x)=\la \pi(x)\psi,\psi \ra, \, \, \forall x \in \cO_d.\]
\end{theorem}

\begin{proof}
First note that such a state $s$ exists using Theorem \ref{theorem: monopartite must yield Cuntz}, since $\varphi$ can be exactly embezzled in a commuting operator framework \cite{CLP17,HP17}. To show uniqueness, suppose that $s:\cO_d \to \cM$ is a state such that $s(v_ixv_j^*)=\delta_{ij}\alpha_i^2 s(x)$ for all $x \in \cO_d$. Consider the GNS representation $\pi:\cO_d \to \bofh$ with cyclic vector $\psi$ such that $s(\cdot)=\la \pi(\cdot)\psi,\psi \ra$, and let $V_j=\pi(v_j)$ for $j=0,...,d-1$. Letting $\omega$ be the marginal of $\psi$ on $\pi(\cO_d)''$, we see that $(R,\psi)$ exactly embezzles $\varphi$ in $(\pi(\cO_d)'',\cH)$ where $R_{ij}=\delta_{j0}V_i^*$. To show uniqueness of $s$, it suffices to show that $\omega$ is unique on $\pi(\cO_d)$. First, suppose that $m>n$, so that there are $\zeta \in \{0,...,d-1\}^n$ and $\xi \in \{0,...,d-1\}^{m-n}$ such that $\mu=\zeta \xi$. Then applying Theorem \ref{theorem: monopartite must yield Cuntz} $n$ times, we see that
\[ \omega(V_{\mu}V_{\nu}^*)=\omega(V_{\mu_1} \cdots V_{\mu_n}V_{\mu_{n+1}} \cdots V_{\mu_m}V_{\nu_n}^* \cdots V_{\nu_1}^*)=\delta_{\zeta,\nu} \alpha_{\nu}^2 \omega(V_{\xi})=0\]
by Lemma \ref{lemma: state on Alice's Od with no stars}. On the other hand, if $m<n$, then $\omega(V_{\mu}V_{\nu}^*)=\overline{\omega((V_{\mu}V_{\nu}^*)^*)}=\overline{\omega(V_{\nu}V_{\mu}^*)}=0$ by the above argument.

Lastly, if $m=n$, then $n$ applications of Theorem \ref{theorem: monopartite must yield Cuntz} shows that $\omega(V_{\mu}V_{\nu}^*)=\delta_{\mu,\nu}\alpha_{\mu}^2$. Extending by linearity and continuity, the state $s$ is uniquely determined on all of $\cO_d$, completing the proof.
\end{proof}

Next, we need a preliminary fact for bipartite exact embezzlement that we will also use later when computing the modular operator and the modular conjugation for $\psi$ (when $\psi$ is cyclic and separating).

\begin{lemma}
\label{lemma: switching W's to V's}
Let $(R,T,\psi)$ be a bipartite exact embezzlement protocol for $\varphi$ in $(\cM,\cM',\cH)$. If $\mu$ and $\nu$ are finite words in $\{0,...,d-1\}$, then
\[ R_{\mu}^*R_{\nu}\psi=\frac{\alpha_{\nu}}{\alpha_{\mu}} T_{\nu}^*T_{\mu}\psi\]
and
\[ T_{\mu}^*T_{\nu}\psi=\frac{\alpha_{\nu}}{\alpha_{\mu}} R_{\nu}^*R_{\mu}\psi.\]
In particular, if $\omega$ (respectively, $\omega'$) is the marginal state of $\psi$ on $\cM$ (respectively, $\cM'$) with support projection $P$ (respectively, $P'$) and if $Q=PP'$ and $V_i=QR_{i0}^*Q$ and $W_j=QT_{j0}^*Q$ for $i,j=0,...,d-1$, then
\[ V_{\mu}V_{\nu}^*\psi=\frac{\alpha_{\nu}}{\alpha_{\mu}} W_{\nu}W_{\mu}^*\psi\]
and
\[ W_{\mu}W_{\nu}^*\psi=\frac{\alpha_{\nu}}{\alpha_{\mu}} V_{\nu}V_{\mu}^*\psi.\]
\end{lemma}

\begin{proof}
Write $\mu=(\mu_1,...,\mu_m)$ and $\nu=(\nu_1,...,\nu_n)$. Then by part (1) of Lemma \ref{lemma: self-testing on subspace}, we have
\begin{align*}
R_{\mu}^*R_{\nu}\psi&=(R_{\mu_m}^* \cdots R_{\mu_1}^*)(R_{\nu_1} \cdots R_{\nu_n})\psi \\
&=\alpha_{\nu_n} T_{\nu_n}^* (R_{\mu_m}^* \cdots R_{\mu_1}^*)(R_{\nu_1} \cdots R_{\nu_{n-1}})\psi.
\end{align*}
Iterating this process through $\nu_1,...,\nu_n$ we see that $R_{\mu}^*R_{\nu}\psi=\alpha_{\nu}T_{\nu}^*R_{\mu}^*\psi$. A similar process then shows that $\alpha_{\nu}R_{\nu}^*R_{\mu}^*\psi=\frac{\alpha_{\nu}}{\alpha_{\mu}}T_{\nu}^*T_{\mu}\psi$, which shows that $R_{\mu}^*R_{\nu}\psi=\frac{\alpha_{\nu}}{\alpha_{\mu}}T_{\nu}^*T_{\mu}\psi$. The argument for showing that $T_{\mu}^*T_{\nu}\psi=\frac{\alpha_{\nu}}{\alpha_{\mu}}R_{\nu}^*R_{\mu}\psi$ is identical. The claim about the operators $V_i$, $W_j$ then follows readily, since $R_{i0}$ commutes with $P'$ since $P' \in \cM$ and $R_{i0}$ commutes with $P$ by Proposition \ref{proposition: P reducing for Ri0s}, so that $Q$ commutes with each $R_{i0}$ (similarly, $Q$ commutes with each $T_{j0}$). Recalling that $Q\psi=\psi$, the result follows.
\end{proof}

For a bipartite exact embezzlement protocol $(R,T,\psi)$ for $\varphi$ in $(\cM,\cM',\cH)$ where $\cM$ (and hence $\cM'$) is in standard form, both the marginal $\omega$ for $\cM$ and $\omega'$ for $\cM'$ are faithful, and so Alice's operators $V_i=R_{i0}^*$ and Bob's operators $W_j=T_{j0}^*$, $i,j=0,...,d-1$, define two collections of $d$ Cuntz isometries satisfying $[V_i,W_j]=0$ for all $i,j$. This yields a representation of the tensor product $\cO_d \otimes_{\max} \cO_d=\cO_d \otimes_{\min} \cO_d$. Let $v_i \otimes 1$, $i=0,...,d-1$ denote a collection of $d$ Cuntz isometries that are generators for $\cO_d \otimes 1$, and let $1 \otimes w_j$, $j=0,...,d-1$ denote a collection of $d$ Cuntz isometries that are generators for $1 \otimes \cO_d$. Then we obtain a unital $*$-homomorphism $\pi:\cO_d \otimes_{\max} \cO_d \to \bofh$ on the (maximal) tensor product of $\cO_d$ with itself such that $\pi(v_i \otimes w_j)=V_iW_j$. Then $\la \pi(\cdot)\psi,\psi \ra$ is a state on $\cO_d \otimes_{\max} \cO_d$. We will show that this state is unique. 

\begin{theorem}
\label{theorem: unique state on Od tensor Od}
There is a unique state $s:\cO_d \otimes_{\max} \cO_d \to \bC$ such that $s(v_i \otimes w_j)=\delta_{ij}\alpha_i$ for each $i,j=0,...,d-1$. Moreover, if $(R,T,\psi)$ is any bipartite exact embezzlement protocol for $\varphi$ in $(\cM,\cM',\cH)$ where the marginals of $\psi$ on $\cM$ and $\cM'$ are faithful, and if $V_i=R_{i0}^*$ and $W_j=T_{j0}^*$ for $i,j=0,...,d-1$, then the unital $*$-homomorphism $\pi:\cO_d \otimes_{\max} \cO_d \to \bofh$ given by $\pi(v_i \otimes 1)=V_i$ and $\pi(1 \otimes w_j)=W_j$ satisfies $s(X)=\la \pi(X)\psi,\psi \ra$ for all $X \in \cO_d \otimes_{\max} \cO_d$.
\end{theorem}

\begin{proof}
It suffices to show that the state $s$ is uniquely determined on each simple tensor of the form $v_{\mu}v_{\nu}^* \otimes w_{\beta} w_{\gamma}^*$, as the span of such tensors is dense in $\cO_d \otimes_{\max} \cO_d$. Note that
\[s(v_{\mu}v_{\nu}^* \otimes w_{\beta}w_{\gamma}^*)=\la V_{\mu}V_{\nu}^*W_{\beta}W_{\gamma}^*\psi,\psi \ra=\frac{\alpha_{\gamma}}{\alpha_{\beta}}\la V_{\mu}V_{\nu}^*V_{\gamma}V_{\beta}^*\psi,\psi \ra\]
by Lemma \ref{lemma: switching W's to V's}. We now have a few cases to consider.

\textbf{Case 1:} $|\nu|=|\gamma|$. Then $V_{\nu}^*V_{\gamma}=\delta_{\nu,\gamma} I$, and applying Theorem \ref{theorem: unique state on Alice's Od} yields
\[ s(v_{\mu}v_{\nu}^* \otimes w_{\beta}w_{\gamma}^*)=\delta_{\nu,\gamma} \frac{\alpha_{\gamma}}{\alpha_{\beta}} \la V_{\mu}V_{\beta}^*\psi,\psi \ra=\delta_{\nu,\gamma}\delta_{\mu,\beta} \alpha_{\gamma}\alpha_{\beta}.\]

\textbf{Case 2:} $|\nu|<|\gamma|$. Then we can find words $\zeta,\xi$ such that $\gamma=\zeta \xi$ and $|\nu|=|\zeta|$. In this case, $V_{\nu}^*V_{\gamma}=\delta_{\nu,\zeta} V_{\xi}$, and we obtain
\[ s(v_{\mu}v_{\nu}^* \otimes w_{\beta}w_{\gamma}^*)=\delta_{\nu,\zeta} \frac{\alpha_{\gamma}}{\alpha_{\beta}} \la V_{\mu\xi} V_{\beta}^*\psi,\psi \ra=\delta_{\nu,\zeta} \delta_{\mu\xi,\beta}\alpha_{\gamma} \alpha_{\beta}.\]

\textbf{Case 3:} $|\nu|>|\gamma|$. Then we can find words $\zeta,\xi$ such that $\nu=\zeta \xi$ and $|\zeta|=|\gamma|$. As $V_{\nu}^*V_{\gamma}=\delta_{\zeta,\gamma} V_{\xi}^*$, we have
\[ s(v_{\mu}v_{\nu}^* \otimes w_{\beta}w_{\gamma}^*)=\delta_{\zeta,\gamma} \frac{\alpha_{\gamma}}{\alpha_{\beta}} \la V_{\mu}V_{\xi}^*V_{\beta}^*\psi,\psi \ra=\delta_{\zeta,\gamma} \frac{\alpha_{\gamma}}{\alpha_{\beta}} \la V_{\mu}V_{\beta \xi}^*\psi,\psi \ra \ra=\delta_{\zeta,\gamma} \delta_{\mu,\beta \xi}\alpha_{\gamma}\alpha_{\beta}\alpha_{\xi}^2. \]

It follows that the state $s$ is uniquely determined by the Schmidt coefficients $\alpha_0,...,\alpha_{d-1}$, and we are done.
\end{proof}

\begin{theorem}
\label{theorem: bipartite yields state on tensor of Cuntz algebra}
For $a=1,2$, let $(R_a,T_a,\psi_a)$ be a bipartite exact embezzlement protocol for $\varphi$ in $(\cM_a,\cM_a',\cH_a)$. Let $\cN_a$ be the von Neumann subalgebra of $\cM_a$ generated by $\{R_{a,i0}: i=0,...,d-1\}$. Let $P_a$ (respectively, $P_a'$) be the support projection of $\omega_a$ (respectively, $\omega_a'$) on $\cN_a$ (respectively, $\cN_a'$) and let $Q_a=P_aP_a'$. Let $V_{a,i}=Q_aR_{a,i0}^*Q_a$ and $W_{a,j}=Q_aT_{a,j0}^*Q_a$ for $i,j=0,...,d-1$. Then there is a unitary $U:Q_1\cH_1 \to Q_2\cH_2$ satisfying
\begin{enumerate}
\item $U\psi_1=\psi_2$,
\item $UV_{1,j}U^*=V_{2,j}$ and $UW_{1,j}U^*=W_{2,j}$ for all $i,j=0,...,d-1$, and
\item $U(Q_1\cN_1 Q_1) U^*=Q_2\cN_2 Q_2$ and $U(Q_1\cN_1'Q_1)U^*=Q_2 \cN_2' Q_2$.
\end{enumerate}
\end{theorem}

\begin{proof}
By Theorem \ref{theorem: unique state on Od tensor Od}, if $\pi_a:\cO_d \otimes_{\max} \cO_d \to \cB(\cH_a)$ is the unital $*$-homomorphism such that $\pi_a(v_i \otimes 1)=V_{a,i}$ and $\pi(1 \otimes w_j)=W_{a,j}$, and if $s_a=\la \pi_a(\cdot)\psi_a,\psi_a\ra$, then $s_1=s_2$. Moreover, since  $\omega_a$ and $\omega_a'$ are faithful on $\cM_a$ and $\cM_a'$ respectively, the unit vector $\psi_a$ is cyclic and separating for $\cM_a$ (and for $\cM_a'$). Thus, $\pi_a(\cO_d \otimes_{\max} \cO_d)\psi$ is dense in $Q_a\cH_a$ since it contains the dense set $\pi_a(\cO_d \otimes 1)\psi$ and $\pi_a(1 \otimes \cO_d)\psi$ (in fact, these sets are equal by Lemma \ref{lemma: switching W's to V's}). Thus, $\pi_a$ is unitarily equivalent to the GNS representation for $s_a$. As $s_1=s_2$ by Theorem \ref{theorem: unique state on Od tensor Od}, there must be a unitary $U:Q_1\cH_1 \to Q_2\cH_2$ such that $U\psi_1=\psi_2$ and $U\pi_1(\cdot)U^*=\pi_2(\cdot)$. This yields the first two claims in the theorem statement. To show that $U(Q_1\cN_1 Q_1)U^*=Q_2\cN_2Q_2$, note that by Proposition \ref{proposition: P reducing for Ri0s}, $P_a\cN_a P_a$ is generated by $\{ P_a R_{i0} P_a: i=0,...,d-1\}$. Since $P_a'$ commutes with each $R_{i0}$ and $P_a$, it is easy to see that $Q_a \cN_a Q_a$ is generated by $\{V_{a,i}: i=0,...,d-1\}$. Since $UV_{1,i}U^*=V_{2,i}$ for all $i$, the result readily follows.
\end{proof}

The monopartite version of Theorem \ref{theorem: bipartite yields state on tensor of Cuntz algebra} follows similarly:

\begin{theorem}
Suppose that, for $a=1,2$, $(R_a,\psi_a)$ is a monopartite exact embezzlement protocol for $\varphi$ in $(\cM_a,\cH_a)$. Let $\cN_a$ be the von Neumann subalgebra of $\cM_a$ generated by $\{R_{a,i0}: i=0,...,d-1\}$. Let $P_a$ (respectively, $P_a'$) be the support projection of $\omega_a$ (respectively, $\omega_a'$) on $\cN_a$ (respectively, $\cN_a'$) and let $Q_a=P_aP_a'$. Let $V_{a,i}=Q_aR_{a,i0}^*Q_a$ for $i=0,...,d-1$. Then there is a unitary $U:Q_1\cH_1 \to Q_2\cH_2$ satisfying
\begin{itemize}
\item $U\psi_1=\psi_2$;
\item $UV_{1,i}U^*=V_{2,i}$ for all $i=0,...,d-1$, and
\item $U(Q_1\cN_1Q_1)U^*=Q_2\cN_2Q_2$.
\end{itemize}
\end{theorem}

\begin{proof}
This follows by using Theorem \ref{theorem: from monopartite to bipartite} to extend $(R_a,\psi_a)$ to a bipartite exact embezzlement protocol $(R_a,T_a,\psi_a)$ in $(\cM_a,\cM_a',\cH_a)$ and applying Theorem \ref{theorem: bipartite yields state on tensor of Cuntz algebra}. (Alternatively, one can directly invoke uniqueness of the state $s$ induced by monopartite exact embezzlement of $\varphi$ from Theorem \ref{theorem: unique state on Alice's Od} and uniqueness of the GNS representation up to unitary equivalence.)
\end{proof}

\begin{remark}
The above theorems show that entanglement embezzlement for $\varphi$ yields a self-test on both the operators involved and the catalyst state $\psi$ that is used, in the sense that the state induced on $\cO_d$ for monopartite embezzlement (and $\cO_d \otimes_{\max} \cO_d$ for bipartite embezzlement) is unique. This concept of describing self-testing in terms of uniqueness of the state was originally proposed in \cite{PSZZ24} in the context of bipartite correlations in the quantum commuting (possibly infinite-dimensional) framework, where one cannot assume a tensor product framework for Alice and Bob's state space (in this setting, the representation $\pi:\cO_d \otimes_{\max} \cO_d \to \bofh$ cannot be decomposed as a tensor product $\pi_1 \otimes \pi_2$ of representations of $\cO_d$). For exact entanglement embezzlement, this is essential, since exact entanglement embezzlement cannot be achieved in any tensor product model \cite{CLP17,HP17}, so the commuting operator framework is needed.
\end{remark}
\section{Type $\text{III}$ factors for exact embezzlement}\label{section: types}

In this section, we will determine the (unique) von Neumann algebra factor that is generated by the collection of $d$ Cuntz isometries involved in monopartite exact embezzlement of $\varphi$ when the state $\psi$ is faithful. We will rely on results in modular theory for von Neumann algebras to determine the type of Alice's (respectively, Bob's) von Neumann algebra. There are many excellent resources on modular theory of von Neumann algebras (such as \cite{Ta03,Su87,St81} to name a few). The reader is directed to these sources and the references therein for more information; we recall the facts that we need below.

Modular theory was a major tool in the classification program for injective von Neumann algebras, and can be carried out for any faithful normal, semifinite weight on a von Neumann algebra $\cM$. In the case where $\cM$ is not separable (that is, does not have separable predual, or equivalently, cannot be represented faithfully and normally on a separable Hilbert space), there is no faithful normal state on $\cM$ and weights are necessary. Fortunately, in our case we are only considering a von Neumann algebra $\cM$ where the catalyst vector $\psi$ induces a faithful state on $\cM$ (which is automatically normal since it is a vector state), so we may assume that $\cM$ is separable below.

Modular theory is carried out when the von Neumann algebra $\cM \subseteq \bofh$ is represented in \textbf{standard form}. We do not need the full strength of standard form of a von Neumann algebra below, so we refer the reader to the work of U. Haagerup \cite{Haa75} on this matter. One setting where a von Neumann algebra $\cM \subseteq \bofh$ is in standard form is when there is a unit vector $\psi \in \cH$ that is cyclic for $\cM$ (that is, $\overline{\cM \psi}=\cH$) and separating for $\cM$ (that is, if $X \in \cM$ and $X\psi=0$, then $X=0$). For bipartite exact embezzlement, after compressing the space and the algebras by the support projections of the marginals on $\cM$ and $\cM'$, we may assume that the marginals $\omega$ and $\omega'$ are faithful on $\cM$ and $\cM'$, respectively, by Proposition \ref{proposition: compressed monopartite embezzlement protocol}, and hence that $\psi$ is cyclic and separating for $\cM$. Thus, we may assume that $\cM \subseteq \bofh$ is in standard form.

We now move to the definition of the modular operator $\Delta_{\psi}$ and the modular conjugation $J$. To start, suppose that $\cM \subseteq \bofh$ is a von Neumann algebra with cyclic and separating vector $\psi$, so that $\cM$ is in standard form. We define the (densely defined, pre-closed) conjugate linear operators $S_0:\cM \psi \to \cM' \psi$ and $F_0:\cM' \psi \to \cM \psi$ by $S_0(X\psi)=X^*\psi$ and $F_0(Y\psi)=Y^*\psi$ for $X \in \cM$ and $Y \in \cM'$. These operators are well-defined and closable. If $S$ is the closure of $S_0$ and $F$ is the closure of $F_0$, then $S^*=F$. The \textbf{modular operator} associated with $\psi$ is the (unbounded) operator $\Delta_{\psi}:=FS=S^*S$. Then the \textbf{modular conjugation} associated with $\psi$ is the conjugate linear isometry $J:\cH \to \cH$ such that $J^2=I$ and $S=J\Delta_{\psi}^{\frac{1}{2}}$.

We first (partially) compute the modular operator and modular conjugation for exact entanglement embezzlement of $\varphi$ on a certain subspace of the ambient Hilbert space. To simplify notation throughout, given a pair $(R,\psi)$ that exactly embezzles $\varphi$ in $(\cM,\cH)$ with $\psi$ a cyclic and separating vector for $\cM$, and given the collection of $d$ Cuntz isometries $V_j=R_{j0}^*$, $j=0,...,d-1$, we identify the Cuntz algebra $\cO_d$ with the $C^*$-subalgebra of $\cM$ generated by $\{V_0,...,V_{d-1}\}$. 

\begin{proposition}
\label{proposition: eigenvalues of modular operator}
Suppose that $(R,\psi)$ exactly embezzles $\varphi$ in $(\cM,\cH)$ and that $\psi$ is cyclic and separating for $\cM$. Write $V_j=R_{j0}^*$ for $0 \leq j \leq d-1$. Let $T \in \cM' \otimes M_d$ be a contraction obtained in Theorem \ref{theorem: from monopartite to bipartite} for which $(R,T,\psi)$ exactly embezzles $\varphi$ in $(\cM,\cM',\cH)$, and let $W_j=T_{j0}^*$. Then for any finite words $\mu,\nu$ in $\{0,...,d-1\}$,
\begin{equation} \Delta_{\psi}(V_{\mu}V_{\nu}^*\psi)=\left(\frac{\alpha_{\mu}}{\alpha_{\nu}}\right)^2 V_{\mu}V_{\nu}^*\psi \label{equation: modular operator}
\end{equation}
and
\begin{equation}
J(V_{\mu}V_{\nu}^*\psi)=W_{\mu}W_{\nu}^*\psi. \label{equation: modular conjugation}
\end{equation}
In particular, the subspace $\overline{\cO_d \psi}$ has an orthonormal basis of eigenvectors for $\Delta_{\psi}$.
\end{proposition}

\begin{proof}
Letting $X=V_{\mu}V_{\nu}^* \in \cM$, we can apply the pre-closed operator $S_0$ and obtain
\begin{equation}
S_0(X\psi)=X^*\psi=V_{\nu}V_{\mu}^*\psi=\frac{\alpha_{\mu}}{\alpha_{\nu}} W_{\mu}W_{\nu}^*\psi \label{equation: pre-closed operator S_0}
\end{equation}
using Lemma \ref{lemma: switching W's to V's}. Since $Y:=\frac{\alpha_{\mu}}{\alpha_{\nu}}W_{\mu}W_{\nu}^* \in \cM'$, we can apply $F_0$ and obtain \[F_0(Y\psi)=Y^*\psi=\frac{\alpha_{\mu}}{\alpha_{\nu}}W_{\nu}W_{\mu}^*\psi=\left(\frac{\alpha_{\mu}}{\alpha_{\nu}}\right)^2 V_{\mu}V_{\nu}^*\psi,\]
where the last step again follows by Lemma \ref{lemma: switching W's to V's}. It follows that $\Delta_{\psi}(V_{\mu}V_{\nu}^*\psi)=\left(\frac{\alpha_{\mu}}{\alpha_{\nu}}\right)^2 V_{\mu}V_{\nu}^*\psi$, so (\ref{equation: pre-closed operator S_0}) holds. Comparing equations (\ref{equation: modular operator}) and (\ref{equation: pre-closed operator S_0}) yields equation (\ref{equation: modular conjugation}).

We now show that $\overline{\cO_d \psi}$ has an orthonormal basis of eigenvectors for $\Delta_{\psi}$. The subspace $\overline{\cO_d \psi}$ is equal to the closed span of the set of all vectors of the form $V_{\mu}V_{\nu}^*\psi$. Note that $V_{\mu}V_{\nu}^* \neq 0$ in $\cO_d$ for each $\mu,\nu$, so $V_{\mu}V_{\nu}^*\psi \neq 0$ since $\psi$ is separating for $\cM$. Hence, each vector $V_{\mu}V_{\nu}^*\psi$ is an eigenvector for $\Delta_{\psi}$ corresponding to the eigenvalue $\left(\frac{\alpha_{\mu}}{\alpha_{\nu}}\right)^2$. Eigenspaces corresponding to distinct eigenvectors for the positive self-adjoint operator $\Delta_{\psi}$ are orthogonal. Hence, one can extract (via Gram-Schmidt) an orthonormal basis for each eigenspace for $\Delta_{\psi}$. The fact that the elements $V_{\mu}V_{\nu}^*\psi$ span a dense set in $\overline{\cO_d \psi}$ shows that $\overline{\cO_d \psi}$ has an orthonormal basis consisting of eigenvectors for $\Delta_{\psi}$.
\end{proof}

One of the significant results in modular theory is that the modular operator induces a continuous family of automorphisms on $\cM$. If $\psi$ is a cyclic and separating vector for $\cM$ in $\cH$, and if $\Delta_{\psi}$ denotes the modular operator corresponding to $\psi$ and $J$ is the modular conjugation corresponding to $\psi$, then the following hold (see any of \cite{St81,Su87,Ta03}):
\begin{itemize}
\item $J\cM J=\cM'$;
\item For each $t \in \bR$, the map $\sigma^{\psi}_t(X)=\Delta_{\psi}^{it} X \Delta_{\psi}^{-it}$ is an automorphism of the von Neumann algebra $\cM$;
\item The map $t \mapsto \sigma^{\psi}_t$ is a strongly continuous homomorphism from the group $(\bR,+)$ into the automorphism group $\text{Aut}(\cM)$.
\end{itemize}
This continuous one-parameter family of automorphisms is often called the \textbf{modular automorphism group} associated with $\psi$. For exact embezzlement, one can partially compute $\sigma^{\psi}_t$ explicitly. Even this partial computation will yield very useful results concerning the von Neumann algebra $\cM$.

\begin{proposition}
\label{proposition: automorphism on Cuntz isometries}
Let $(R,\psi)$ be a monopartite exact embezzlement protocol for $\varphi$ in $(\cM,\cH)$ where $\psi$ is cyclic and separating for $\cM$. Let $V_j=R_{j0}^*$. Then for each $j=0,...,d-1$,  $\sigma^{\psi}_t(V_j)=\alpha_j^{-2it} V_j$.
\end{proposition}

\begin{proof}
We compute, by Proposition \ref{proposition: eigenvalues of modular operator},
\[(\sigma^{\psi}_t(V_j))\psi=\Delta_{\psi}^{it}V_j \Delta_{\psi}^{-it}\psi=\Delta_{\psi}^{it}(V_j\psi)=\alpha_j^{-2it} V_j \psi.\]
Since $\psi$ is separating for $\cM$, it follows that $\sigma^{\psi}_t(V_j)=\alpha_j^{-2it}V_j$.
\end{proof}

Our first goal is to show that this modular automorphism group cannot be unitarily implemented; in other words, the automorphism group is not inner. Since all modular automorphisms of semifinite von Neumann algebras must be inner \cite{Co73}, this shows that $\cM$ cannot be semifinite. To prove this fact, we use the following lemma that gives a necessary condition for the modular automorphism $\sigma^{\psi}_t$ to be inner for a particular value of $t$. By way of notation, given the state $\varphi \in \bC^d \otimes \bC^d$ with Schmidt coefficients $\alpha_0 \geq \cdots \geq \alpha_{d-1}>0$, we define $H_{\alpha_j}$ to be the cyclic subgroup of $(\bR,+)$ generated by $\frac{2\pi}{-\log(\alpha_j^2)}=-\frac{\pi}{\log(\alpha_j)}$, and we define $H_{\varphi}=\bigcap_{j=0}^{d-1} H_{\alpha_j}$.

\begin{lemma}
\label{lemma: state on unitarily implemented automorphism}
Let $(R,\psi)$ be a monopartite exact embezzlement protocol for $\varphi$ in $(\cM,\cH)$ with $\psi$ separating for $\cM$. Let $\omega$ be the marginal of $\psi$ on $\cM$, and let $V_j=R_{j0}^*$ for $j=0,...,d-1$. Let $t \in \bR$, and suppose that $u(t) \in \cM$ is a unitary such that $u(t)V_ju(t)^*=\alpha_j^{-2it}V_j$ for all $j$. If $\omega(u(t)) \neq 0$, then $t \in H_{\varphi}$.
\end{lemma}

\begin{proof}
Pre-multiplying by $u(t)^*$ and post-multiplying by $V_j^*$ yields $V_ju(t)^*V_j^*=\alpha_j^{-2it} u(t)^* V_jV_j^*$. Applying $\omega$ yields
\[ \alpha_j^2 \omega(u(t)^*)=\omega(V_ju(t)^*V_j^*)=\alpha_j^{-2it}\omega(u(t)^*V_jV_j^*).\]
Thus, $\alpha_j^{2it+2} \omega(u(t)^*)=\omega(u(t)^*V_jV_j^*)$. Summing over $j$,
\[ \sum_{j=0}^{d-1} \alpha_j^{2it+2}\omega(u(t)^*)=\sum_{j=0}^{d-1} \omega(u(t)^*V_jV_j^*)=\omega(u(t)^*).\]
If $\omega(u(t))=0$ then $\omega(u(t)^*)=\overline{\omega(u(t))}=0$. If $\omega(u(t)) \neq 0$, then $\omega(u(t)^*)=\overline{\omega(u(t))} \neq 0$, and we must have
\[ 1=\sum_{j=0}^{d-1} \alpha_j^{2it+2}=\left\la \begin{pmatrix} \alpha_0^{2it+1} \\ \vdots \\ \alpha_{d-1}^{2it+1} \end{pmatrix},\begin{pmatrix} \alpha_0 \\ \vdots \\ \alpha_{d-1} \end{pmatrix}\right\ra.\]
Both vectors in this inner product are unit vectors since $\sum_{j=0}^{d-1} \alpha_j^2=1$. By Proposition \ref{proposition: equality of CS} (with both operators being the identity), we must have $\alpha_j^{2it+1}=\alpha_j$ for every $j$, forcing $\alpha_j^{2it}=1$ for each $j$. This forces $t \in H_{\alpha_j}$ for each $j$, so that $t \in H_{\varphi}$.
\end{proof}

In the case when $\psi$ is cyclic and separating for $\cM$ and $\cM$ is generated by the collection of $d$ Cuntz isometries $\{V_j\}_{j=0}^{d-1}$, it is easy to see that $\sigma^{\psi}_t=\text{id}$ whenever $t \in H_{\varphi}$. As it turns out, in this case $\sigma^{\psi}_t$ will be trivial (hence inner) if and only if $t \in H_{\varphi}$ (see Remark \ref{remark: converse to lemma on inner}). In general, even if $\cM$ is not generated by the Cuntz isometries in the monopartite exact embezzlement protocol, we still have the following.

\begin{theorem}
If $(R,\psi)$ exactly embezzles $\varphi$ in $(\cM,\cH)$ where $\psi$ is cyclic and separating for $\cM$, then $\cM$ is not semifinite.
\end{theorem}

\begin{proof}
If $\cM$ were semifinite, then $\sigma^{\psi}_t$ would be inner for all $t$ by the Connes cocycle derivative theorem \cite[Theorem~1.3.4]{Co73}. Moreover, in this case there exists a strongly continuous group homomorphism from $\bR$ to the unitary group $\cU(\cM)$ of $\cM$ given by $t \mapsto u(t) \in \cU(\cM)$ such that $\sigma^{\psi}_t(\cdot)=u(t)(\cdot)u(-t)$ for all $t$ \cite[Theorem~1.2.1]{Co73}. By Lemma \ref{lemma: state on unitarily implemented automorphism}, $\omega(u(t))=0$ for all $t$ except possibly on a countable subset of $\bR$. As $t \mapsto u(t)$ is strongly continuous, the map $t \mapsto \omega(u(t))$ is continuous, so we must have $\omega(u(t))=0$ for all $t \in \bR$. But this is absurd since $u(0)=1$ and $\omega(u(0))=\omega(1)=1$. Thus, $\cM$ is not semifinite.
\end{proof}

\begin{corollary}
\label{corollary: monopartite has corner of type III}
If there is a monopartite exact embezzlement protocol $(R,\psi)$ in $(\cM,\cH)$, and if $P$ (respectively, $P'$) denotes the support projection of $\omega=\la (\cdot)\psi,\psi \ra$ on $\cM$ (respectively, $\omega'$ on $\cM'$) and $Q=PP'$, then $Q\cM Q$ is Type III.
\end{corollary}

\begin{corollary}
\label{corollary: monopartite is zero on semifinite part}
Let $\cM \subseteq \bofh$ be a von Neumann algebra, and write $\cM=\cM_0 \oplus \cM_1$ where $\cM_0$ is Type III and $\cM_1$ is semifinite. Suppose that $(R,\psi)$ exactly embezzles $\varphi$ in $(\cM,\cH)$. Let $P$ (respectively, $P'$) be the support projection in $\cM$ (respectively, $\cM'$) of the marginal $\omega$ on $\cM$ (respectively, $\omega'$ on $\cM'$) and let $Q=PP'$. Then $Q\cM_1 Q=\{0\}$.
\end{corollary}

The next goal is to determine the type of the subalgebra of $\cM$ generated by Alice's blocks that actually contribute to exact embezzlement.  Using the modular operator $\Delta_{\psi}$ that we already computed, we can describe the von Neumann algebra generated by $\mathcal{O}_d$. We assume throughout that the catalyst state vector $\psi$ is separating (i.e. that $\omega$ is faithful), in which case the von Neumann algebra generated by $\mathcal{O}_d$ is exactly $\mathcal{O}_d''$. Note that, since $\mathcal{O}_d$ is a finitely generated nuclear $C^*$-algebra and is simple, the double commutant $\mathcal{O}_d''$ is a factor with separable pre-dual and is injective (hence approximately finite-dimensional, or AFD for short, by a deep theorem of Connes \cite{Co76}). As a result, $\mathcal{O}_d''$ can be determined by the classification of separable AFD Type $\text{III}$ factors. We will show that computing the type of $\mathcal{O}_d''$ can be done by considering the spectrum of $\Delta_{\psi}$.

For the next result, we use the notation $\sigma(X)$ for the spectrum of a (possibly unbounded) operator. Given the Schmidt coefficients $\alpha_0,...,\alpha_{d-1}$ of $\varphi$, we let $G_{\varphi}$ be the closed subgroup of $(\bR^+, \times)$ generated by $\{\alpha_0^2,...,\alpha_{d-1}^2\}$.

\begin{lemma}
\label{lemma: positive spectrum of modular operator}
Let $(R,\psi)$ be a monopartite exact embezzlement protocol for $\varphi$ in $(\cM,\cH)$, where $\psi$ is cyclic and separating for $\cM$ and $\cM$ is generated by $\{V_j: j=0,...,d-1\}$ where $V_j=R_{j0}^*$. Then $\sigma(\Delta_{\psi} \setminus \{0\})=G_{\varphi}$.
\end{lemma}

\begin{proof}
Clearly by Proposition \ref{proposition: eigenvalues of modular operator} the subgroup generated by $\{ \alpha_0^2,...,\alpha_{d-1}^2\}$ is contained in $\sigma(\Delta_{\psi})$; as the spectrum is closed (even for unbounded operators), we see that $G_{\varphi} \subseteq \sigma(\Delta_{\psi}) \setminus \{0\}$. For the converse direction, suppose that $\lambda \neq 0$ and that $\lambda \not\in G_{\varphi}$. Then $\text{dist}(\lambda,G_{\varphi} \cup \{0\})=r>0$. Since $\cM$ is generated by the operators $V_j$, the $C^*$-algebra $\cO_d=C^*(V_0,...,V_{d-1})$ is SOT dense in $\cM$, so $\overline{\cO_d \psi}=\overline{\cM \psi}=\cH$. By Proposition \ref{proposition: eigenvalues of modular operator}, there is an orthonormal basis $(f_n)_{n=1}^{\infty}$ for $\cH=\overline{\cM \psi}$ consisting of eigenvectors for $\Delta_{\psi}$ with eigenvalues belonging to $G_{\varphi}$. (The orthonormal basis is countable since $\cM$ is separable.) Let $\lambda_n$ be the eigenvalue of $\Delta_{\psi}$ for the eigenvector $f_n$. Since $\lambda_n-\lambda \geq r$ for all $n$, the diagonal operator $T$ given by $T(f_n)=\frac{1}{\lambda_n-\lambda} f_n$ extends to a bounded linear operator on $\cH$ and is the inverse of $\Delta_{\psi}-\lambda I$ on its domain, so $\lambda \not\in \sigma(\Delta_{\psi})$.
\end{proof}

The next lemma shows that, when Alice's algebra $\cM$ is generated by $\{R_{i0}: i=0,...,d-1\}$ and in standard form, then $\cM'$ is generated by Bob's Cuntz isometries.
\begin{lemma}
\label{lemma: only one projection for standard form for minimal algebra}
Suppose that $(R,T,\psi)$ exactly embezzles $\varphi$ in $(\cM,\cM',\cH)$. Let $\cN_A$ be the von Neumann subalgebra generated by $\{R_{i0}: i=0,...,d-1\}$ and let $\cN_B$ be the von Neumann subalgebra generated by $\{T_{j0}: j=0,...,d-1\}$. Let $P$ (respectively, $P'$) be the support projection of the marginal state $\omega$ on $\cN_A$ (respectively, $\omega'$ on $\cN_A'$), and let $Q=PP'$. Then $(Q\cN_A Q)'=Q\cN_B Q$.
\end{lemma}

\begin{proof}
Let $V_i=QR_{i0}^*Q$ and $W_j=QT_{j0}^*Q$ for $i,j=0,...,d-1$. Note that $P$ commutes with each $R_{i0}$ by Proposition \ref{proposition: P reducing for Ri0s}, as does $P'$ since $P' \in \cN_A'$, so $Q\cN_A Q$ is generated as a von Neumann algebra by $\{V_0,...,V_{d-1}\}$; similarly, $Q\cN_B Q$ is generated as a von Neumann algebra by $\{W_0,...,W_{d-1}\}$. 

Since $Q\cN_A Q$ is in standard form on $Q\cH$ with cyclic separating vector $\psi$, the modular conjugation $J:Q\cH \to Q\cH$ satisfies $J(Q\cN_A Q)J=(Q\cN_A Q)'=Q\cN_A' Q$. Using Proposition \ref{proposition: eigenvalues of modular operator}, we have $J(V_{\mu}V_{\nu}^*\psi)=W_{\mu}W_{\nu}^*\psi$ for all finite words $\mu,\nu$ in $\{0,...,d-1\}$. Note, then, that
\[ J(V_{\mu}V_{\nu}^*)J\psi=J(V_{\mu}V_{\nu}^*\psi)=W_{\mu}W_{\nu}^*\psi.\]
Since $\psi$ is cyclic for $Q\cN_A Q$, $\psi$ is separating for $Q\cN_A' Q$, hence separating for $Q\cN_B Q$ since $\cN_B \subseteq \cN_A$. It follows that $J(V_{\mu}V_{\nu}^*)J=W_{\mu}W_{\nu}^*$. A standard argument involving SOT limits then shows that $J(Q\cN_A Q)J=Q\cN_B Q$. It follows that $Q\cN_B Q=Q\cN_A' Q=(Q\cN_A Q)'$.
\end{proof}

The next theorem shows exactly which of the separable AFD Type $\text{III}$ factors can appear as ``smallest" observable algebras for exact entanglement embezzlement. As the state $\omega$ on $\cO_d''$ is a quasi-free state, work of Izumi \cite{Iz93} yields the appropriate factor. For the sake of the reader, we include a sketch of the argument.

\begin{theorem}
\label{theorem: the type}
Suppose that $(R,\psi)$ exactly embezzles $\varphi$ in $(\cM,\cH)$. Let $\cN$ be the von Neumann algebra generated by $\{R_{i0}: i=0,...,d-1\}$. Let $P$ (respectively, $P'$) be the support projection of the marginal state $\omega$ of $\psi$ in $\cN$ (respectively, of the marginal state $\omega'$ of $\psi$ in $\cN'$), and let $Q=PP'$. Let $G_{\varphi}$ be the closed subgroup of $(\bR^+,\times)$ generated by $\{ \alpha_0^2,...,\alpha_{d-1}^2 \}$.
\begin{enumerate}
\item If $G_{\varphi}$ is countable, then $Q\cN Q$ is isomorphic to the unique (separable) AFD Type $\text{III}_{\lambda}$ factor, where $\lambda=\sup(G_{\varphi} \cap (0,1))=\max(G_{\varphi} \cap (0,1))$. Moreover, $\lambda$ is a root of a polynomial equation of the form $x^{m_0}+\cdots+x^{m_{d-1}}-1=0$ for certain $m_0,...,m_{d-1} \in \bN$.
\item Otherwise, $G_{\varphi}=\bR^+$ and $Q\cN Q$ is isomorphic to the unique (separable) AFD Type $\text{III}_1$ factor.
\end{enumerate}
\end{theorem}

\begin{proof}
Let $\cN_0=Q\cN Q$, so that $\cN_0$ is in standard form on $Q\cH$ with cyclic separating vector $\psi$. Let $V_i=QR_{i0}^*Q$, which define a collection of $d$ Cuntz isometries in $\cN_0$. We write $\cO_d$ for $C^*(\{V_0,...,V_{d-1}\})$; then $\cN_0=\cO_d''$ (where the double commutant is with respect to the representation of $\cO_d \subseteq \cB(Q\cH)$). Let $\cO_{U(d)}$ be the $C^*$-subalgebra of $\cO_d$ consisting of all fixed points under the action $U(d) \curvearrowright \cO_d$ given by $\gamma_U(V_i)=\sum_{j=0}^{d-1} U_{ji}V_j$ for $U$ in the unitary group $U(d)$ of $M_d(\bC)$. Then $\cO_d'' \cap (\cO_{U(d)})'=\bC 1$ \cite[Proposition~4.5]{Iz93}, while it is easy to see that $\cO_{U(d)}''$ is contained in the fixed point algebra $(\cO_d'')^{\sigma^{\psi}}$ of the modular automorphism group $\{ \sigma^{\psi}_t: t\in \bR\}$ on $\cN_0$. It follows that $(\cO_d'')^{\sigma^{\psi}}=\cN_0^{\sigma^{\psi}}$ is a factor. Thus, the Connes spectrum $\Gamma(\cN)$ (see \cite{Co73}) is exactly the Arveson spectrum $\text{Sp}(\sigma^{\psi})$ of the modular automorphism group $\{ \sigma^{\psi}_t: t \in \bR\}$ \cite{Co73}. Moreover, the Arveson spectrum of the modular automorphism group of any faithful normal semifinite weight on $\cN_0$ is the positive part of the spectrum of the modular operator. It follows that $\Gamma(\cN_0)=\sigma(\Delta_{\psi}) \cap \bR^+=G_{\varphi}$ by Lemma \ref{lemma: positive spectrum of modular operator}.

If (1) holds, then $G_{\varphi}$ is cyclic. As $G_{\varphi}$ is a closed subgroup of $(\bR^+,\times)$, it is well-known that $1$ is an isolated point in $G_{\varphi}$ (since otherwise we would have $G_{\varphi}=\bR^+$). Hence, the unique generator $\lambda \in (0,1)$ for $G_{\varphi}$ must be the maximal element in $G_{\varphi} \cap (0,1)$. Thus, for each $j$, $\alpha_j^2 \in \{ \lambda^n: n \in \bN\}$ since each $\alpha_j^2 \in (0,1)$. By classification of Type $\text{III}$ factors \cite{Co73}, since $\Gamma(\cN_0)=\{ \lambda^n: n \in \bZ\}$, it follows that $\cN_0$ is the unique separable AFD Type $\text{III}_{\lambda}$ factor. In this case, there exists an $m_j \in \bN$ such that $\alpha_j^2=\lambda^{m_j}$. Using the fact that $\sum_{j=0}^{d-1} \alpha_j^2=1$, we arrive at the polynomial equation $\lambda^{m_0}+\cdots+\lambda^{m_{d-1}}-1=0$, as desired.

If (1) does not hold, then $G_{\varphi}$, being an uncountable closed subgroup of $\bR^+$, must be $\bR^+$ itself. By the classification of AFD Type $\text{III}$ factors, $\cN_0$ is the unique separable AFD Type $\text{III}_1$ factor \cite{Haa87}, completing the proof.
\end{proof}

Another helpful characterization of when the factor obtained is Type $\text{III}_1$ is given by the following fact.

\begin{corollary}
\label{corollary: type III1 if and only if log ratios are not rational}
Let $(R,\psi)$ be a monopartite exact embezzlement protocol for $\varphi$ in $(\cM,\cH)$. Let $\cN=W^*(\{R_{i0}: i=0,...,d-1\})$, and let $P \in \cN$ (respectively, $P' \in \cN'$) be the support projection of the marginal $\omega$ for $\psi$ in $\cN$ (respectively, the marginal $\omega'$ for $\psi$ in $\cN'$). Let $Q=PP'$. Then:
\begin{enumerate}
\item If $\displaystyle \frac{\ln(\alpha_i)}{\ln(\alpha_j)} \in \mathbb{Q}$ for all $i,j=0,...,d-1$, then $Q\cN Q$ is Type $\text{III}_{\lambda}$ for some $\lambda \in (0,1)$.
\item If $\displaystyle \frac{\ln(\alpha_i)}{\ln(\alpha_j)} \not\in \mathbb{Q}$ for some $i,j=0,...,d-1$, then $Q\cN Q$ is Type $\text{III}_1$.
\end{enumerate}
\end{corollary}

\begin{proof}
This is a consequence of the well-known fact that the closed subgroup of $\bR^+$ generated by $\alpha_0^2,...,\alpha_{d-1}^2$ is countable if and only if $\frac{\ln(\alpha_i^2)}{\ln(\alpha_j^2)}=\frac{\ln(\alpha_i)}{\ln(\alpha_j)}$ is rational for all choices of $i,j$.
\end{proof}

\begin{remark}
\label{remark: converse to lemma on inner}
Theorem \ref{theorem: the type} also provides a partial converse to Lemma \ref{lemma: state on unitarily implemented automorphism} in the case when $\cM=\cO_d''$. Indeed, by Theorem \ref{theorem: the type}, $\sigma^{\psi}_t$ is inner if and only if $t \in H_{\alpha}$, and for precisely those values of $t$ we have $\sigma^{\psi}_t=\text{id}$ by an application of Proposition \ref{proposition: automorphism on Cuntz isometries}, so $u(t)=1$ for $t \in H_{\alpha}$ and $\omega(u(t))=1$. In fact, if $\cM=\cO_d''$, then for each $t \in \bR$, the automorphism $\sigma^{\psi}_t$ is either trivial on $\cM$ or not inner.
\end{remark}

In light of Theorem \ref{theorem: the type}, we make the following definition.

\begin{definition}
Let $d \geq 2$ and let $\displaystyle \varphi=\sum_{j=0}^{d-1} \alpha_j e_j \otimes e_j \in \bC^d \otimes \bC^d$, where $\alpha_0 \geq \cdots \geq \alpha_{d-1}>0$ with $\displaystyle \sum_{j=0}^{d-1} \alpha_j^2=1$. The \textbf{minimal embezzling factor for $\varphi$} is the unique (separable) AFD Type $\text{III}_{\lambda}$ factor $\cM$, $\lambda \in (0,1]$, generated by the entries $\{R_{i0}: i=0,...,d-1\}$ of a monopartite exact embezzlement protocol $(R,\psi)$ for $\varphi$ with $\psi$ cyclic and separating for
$\cM=W^*(\{R_{i0}: i=0,...,d-1\})$ as in Theorem \ref{theorem: the type}.
\end{definition}

\begin{example}
For the maximally entangled Bell state $\varphi=\sum_{j=0}^{d-1} \frac{1}{\sqrt{d}} e_0 \otimes e_0$, the squares of the Schmidt coefficients for $\varphi$ are all $\frac{1}{d}$, so Theorem \ref{theorem: the type} shows that the minimal embezzling factor for $\varphi$ is the unique AFD Type $\text{III}_{\frac{1}{d}}$ factor.
\end{example}

We use Theorem \ref{theorem: the type} to summarize which Type $\text{III}_{\lambda}$ factors are minimal embezzling factors. To do this, we need a simple lemma.

\begin{lemma}
\label{lemma: polynomials arising must have unique root between 0 and 1}
Let $d \in \bN$ with $d \geq 2$. Suppose that $p$ is a polynomial in $\bZ[x]$ such that $p(0)=-1$, $p(1)=d-1$ and such that $p$ has no negative coefficients except the constant coefficient. Then $p$ has exactly one root in $(0,1)$.
\end{lemma}

\begin{proof}
The existence of a root follows from applying the intermediate value theorem to $p(x)$ on the interval $[0,1]$, since $p(0)=-1<0$ and $p(1)=d-1>0$. For uniqueness, if there were two roots $\lambda_1<\lambda_2$ for $p(x)$ in $(0,1)$, then by the mean value theorem there would exist a $c \in (\lambda_1,\lambda_2) \subseteq (0,1)$ such that $p'(c)=0$. But since $p$ is not constant and has all non-negative coefficients except the constant coefficient, it is easy to see that $p'(x)>0$ for all $x>0$, which is a contradiction. Thus, the root $\lambda \in (0,1)$ is unique for $p$.
\end{proof}

Before stating which Type $\text{III}_{\lambda}$ factors appear as minimal embezzling factors, we adopt a bit of terminology. For a polynomial $p(x)=\sum_{i=0}^n b_i x^i$ with $b_0,...,b_n \in \bZ$, the \textbf{exponent list} of $p$ is the set $\{i: b_i \neq 0\}$.
\begin{theorem}
\label{theorem: minimal embezzling factors}
Let $d \geq 2$, and let $\Lambda_d$ be the set of all $\lambda \in (0,1]$ for which there exists a state $\varphi \in \bC^d \otimes \bC^d$ with all non-zero Schmidt coefficients whose minimal embezzling factor is Type $\text{III}_{\lambda}$. Then:
\begin{enumerate}
\item $1 \in \Lambda_d$.
\item For $\lambda \in (0,1)$, we have that $\lambda \in \Lambda_d$ if and only if $\lambda$ is the root of a polynomial $p(x)\in \bZ[x]$ such that
\begin{itemize}
\item $p(0)=-1$;
\item $p(1)=d-1$;
\item All non-constant coefficients of $p$ are non-negative; and
\item The exponent list of $p$ is coprime.
\end{itemize}
\end{enumerate}
\end{theorem}

\begin{proof}
To prove (1), note that if $d=2$ the state $\varphi$ with Schmidt coefficients $\frac{\sqrt{3}}{2}$, $\frac{1}{2}$ has minimal embezzling factor of Type $\text{III}_1$, since $\frac{\ln(\sqrt{3}/2)}{\ln(1/2)}=\frac{\frac{1}{2}\ln(3)-\ln(2)}{-\ln(2)}=-\frac{\ln(3)}{2\ln(2)}+1$ is not rational. If $d \geq 3$, then one can use a state $\varphi$ with one Schmidt coefficient $\frac{1}{\sqrt{2}}$, one equal to $\frac{1}{\sqrt{3}}$, and the rest chosen non-zero so that the Schmidt coefficients have squares summing to $1$ (this works since $\frac{\ln(1/\sqrt{2})}{\ln(1/\sqrt{3})}=\frac{\ln(2)}{\ln(3)}$ is not rational). Thus, (1) holds.

To show (2), suppose that $\lambda \in (0,1) \cap \Lambda_d$, corresponding to a state $\varphi \in \bC^d \otimes \bC^d$ with Schmidt coefficients $\alpha_0 \geq \alpha_1 \geq \cdots \geq \alpha_{d-1}$. Then by Theorem \ref{theorem: the type}, $\lambda$ is a root of a polynomial of the form $p(x)=x^{m_0}+\cdots+x^{m_{d-1}}-1$ for some $m_0,...,m_{d-1} \in \bN$, where $\alpha_j^2=\lambda^{m_j}$ for each $j$ . Such a polynomial has all non-negative coefficients (except the constant coefficient) and satisfies $p(0)=-1$ and $p(1)=d-1$. Moreover, since $\lambda \in \langle \alpha_0^2,...,\alpha_{d-1}^2\rangle$, there exist $b_0,...,b_{d-1} \in \bZ$ such that
\[ \lambda=(\alpha_0^2)^{b_0} \cdots (\alpha_{d-1}^2)^{b_{d-1}}=\lambda^{m_0b_0+\cdots+m_{d-1}b_{d-1}}.\]
Taking the natural logarithm and noting that $\lambda \neq 1$ (so $\ln(\lambda) \neq 0$), it follows that $m_0b_0+\cdots+m_{d-1}b_{d-1}=1$, so that $\text{gcd}(m_0,...,m_{d-1})=1$.

Conversely, suppose that $0<\lambda<1$ is a root of a polynomial $p(x)$ that has all the prescribed properties in (2). Then we may write $p(x)=x^{m_0}+\cdots+x^{m_{d-1}}-1$ for some $m_0,...,m_{d-1} \in \bN$, with $\text{gcd}(m_0,...,m_{d-1})=1$. Define $\alpha_j=\lambda^{m_j/2}$ for each $j=0,...,d-1$. Each $\alpha_j$ is positive and $\sum_{j=0}^{d-1} \alpha_j^2=\sum_{j=0}^{d-1} \lambda^{m_j}=1$. As $\text{gcd}(m_0,...,m_{d-1})=1$, there are integers $b_0,...,b_{d-1}$ such that $m_0b_0+\cdots+m_{d-1}b_{d-1}=1$, so that
\[ \lambda=\lambda^{m_0b_0 + \cdots + m_{d-1}b_{d-1}}=(\alpha_0^2)^{b_0} \cdots (\alpha_{d-1}^2)^{b_{d-1}}.\]
Thus, $\lambda$ is in the multiplicative subgroup of $\bR^+$ generated by $\{\alpha_0^2,...,\alpha_{d-1}^2\}$ and generates this subgroup. By Theorem \ref{theorem: the type}, in the case of exact embezzlement to the state $\varphi=\sum_{j=0}^{d-1} \alpha_j e_j \otimes e_j$, since $G_{\varphi}$ is generated by $\lambda$, one has that $\mathcal{O}_d''$ is Type $\text{III}_{\lambda}$. Hence, $\lambda \in \Lambda_d$, so (2) holds.

To show (3), we note that every $\lambda \in \Lambda_d$ must be an algebraic number by Theorem \ref{theorem: the type}, so $\Lambda_d$ is at most countably infinite. To show that $\Lambda_d$ is infinite for each $d \geq 2$, consider the polynomial $p_{m,d}(x)=x^{2m-1}+(d-1)x^2-1$ for each $m \in \bN$. Then $p(0)=-1$, $p(1)=d-1$ and $p$ has all non-negative coefficients except the constant coefficient while $\text{gcd}(2m-1,2)=1$, so by Lemma \ref{lemma: polynomials arising must have unique root between 0 and 1} and by (2), the unique root $\lambda_{m,d}$ of $p_{m,d}(x)$ between $0$ and $1$ belongs to $\Lambda_d$. If $m<n$ in $\bN$, then $p_{n,d}(x)-p_{m,d}(x)=x^{2n-1}-x^{2m-1}=x^{2m-1}(x^{2n-2m}-1)$ only has real roots in the set $\{-1,0,1\}$, so $p_{n,d}$ and $p_{m,d}$ cannot have a common root in $(0,1)$. This shows that $\lambda_{m,d} \neq \lambda_{n,d}$ whenever $n \neq m$, so $\Lambda_d$ is infinite.
\end{proof}

As a consequence of Theorem \ref{theorem: the type}, uncountably many of the separable AFD Type $\text{III}$ factors do not arise as minimal embezzling factors, in the context of exact embezzlement (in particular, any of the uncountably many non-isomorphic Type $\text{III}_0$ factors, or any of the Type $\text{III}_{\lambda}$ factors where $\lambda$ is transcendental). We now show that $\bigcup_{d=2}^{\infty} \Lambda_d$ even misses infinitely many algebraic numbers.

\begin{corollary}
There are a countably infinite number of algebraic numbers $\lambda \in (0,1)$ for which $\lambda\not\in \bigcup_{d=2}^{\infty} \Lambda_d$--in other words, for which the AFD Type $\text{III}_{\lambda}$ factor is not a minimal embezzling factor.
\end{corollary}

\begin{proof}
We need only construct an infinite set of algebraic numbers in $(0,1) \setminus \bigcup_{d=2}^{\infty} \Lambda_d$. To do this, note that any algebraic number $\beta \in (0,1)$ whose minimal polynomial over $\mathbb{Q}$ has at least two distinct roots in $(0,1)$ will not belong to $\bigcup_{d \geq 2} \Lambda_d$ by Lemma \ref{lemma: polynomials arising must have unique root between 0 and 1} and Theorem \ref{theorem: minimal embezzling factors}. A countably infinite family of such numbers are those of the form $\lambda_q=\displaystyle\frac{1}{2}+\frac{1}{\sqrt{q}}$, where $q$ is prime and $q \geq 5$. Then $\lambda_q \in (0,1) \setminus \mathbb{Q}$, and the minimal polynomial of $\lambda_q$ will have two roots in $(0,1)$, namely $\frac{1}{2} \pm \frac{1}{\sqrt{q}}$. Hence, $\lambda_q \not\in \bigcup_{d=2}^{\infty} \Lambda_d$. Clearly $\lambda_{q_1} \neq \lambda_{q_2}$ if $q_1,q_2$ are distinct primes at least $5$, so we are done.
\end{proof}

On the other hand, approximate embezzling states exist in all AFD Type $\text{III}_{\lambda}$ factors for $\lambda \in (0,1]$, and also exist in some (but not all) AFD Type $\text{III}_0$ factors \cite{LSWW24}. Our results are only looking at the ``smallest" observable algebras possible for a player in exactly embezzling one entangled state, so this does not contradict \cite{LSWW24}. 

We close with the following result, which follows in a similar way to \cite{LSWW24} (but is slightly more generally stated here).

\begin{corollary}
\label{corollary: simultaneous embezzlement only on countable family}
Let $\cM \subseteq \bofh$ be a von Neumann algebra with $\cH$ separable. Let $d \in \bN_{\geq 2}$, and suppose that $\psi \in \cH$ is a unit vector. Let $\cE(\cM,\psi,d)$ be the set of all $d$-tuples of Schmidt coefficients for which there is a contraction $R \in M_d \otimes \cM$ and a state $\varphi \in \bC^d \otimes \bC^d$ with Schmidt coefficients $\alpha_0,...,\alpha_{d-1}$, such that $(R,\psi)$ exactly embezzles $\varphi$ in $(\cM,\cH)$. Then $\cE(\cM,\psi,d)$ is countable. 
\end{corollary}

\begin{proof}
By Proposition \ref{proposition: eigenvalues of modular operator}, if $(\alpha_0,...,\alpha_{d-1}) \in \cE(\cM,\psi,d)$, then each of $\alpha_0^2,...,\alpha_{d-1}^2$ appear as eigenvalues of the modular operator $\Delta_{\psi}$. For self-adjoint positive operators, eigenvectors corresponding to distinct eigenvalues must be orthogonal. Since $\cH$ is separable, this forces $\cE(\cM,\psi,d)$ to be countable.
\end{proof}

In earlier work \cite{Li25}, L. Liu proved that there is a faithful normal state $\omega$ on a von Neumann algebra $\cM$ with separable predual that is simultaneously an embezzler for a (necessarily countable) dense subset of all states in $\bC^{2^n} \otimes \bC^{2^n}$. Corollary \ref{corollary: simultaneous embezzlement only on countable family} proves that this is the best that one can achieve when the von Neumann algebra has separable predual. The work of \cite{LSWW24} shows that universal exact embezzlers exist, but necessarily these must occur in non-separable Hilbert spaces.

\section*{Acknowledgements}

Part of this research was carried out while visiting the Mittag-Leffler Institute as part of the ``Operator Algebras and Quantum Information" 2026 workshop. The author thanks the institute for their hospitality. The author gratefully acknowledges support for this visit from NSF grant DMS-2454136. The author also thanks Richard Cleve, Li Liu and Vern Paulsen for helpful conversations. We also thank Alexander Stottmeister for constructive feedback on an earlier version of this paper.

\end{document}